\documentclass[12pt]{amsart}
\usepackage{amsmath}
\usepackage{amsfonts}
\usepackage{amssymb}
\usepackage{mathtools}
\usepackage{bbding}
\usepackage{stmaryrd}
\usepackage{txfonts}
\usepackage{graphicx}
\usepackage{epsfig}
\usepackage{xypic}
\usepackage{tikz}
\usepackage{longtable}
\usepackage[all]{xy}
\usepackage{pgflibraryarrows}
\usepackage{pgflibrarysnakes}
\usepackage[shortlabels]{enumitem}
\usepackage{ifpdf}
\ifpdf
\usepackage[colorlinks,final,backref=page,hyperindex]{hyperref}
\else
\usepackage[colorlinks,final,backref=page,hyperindex,hypertex]{hyperref}
\fi
\usepackage{graphicx}
\usepackage{epstopdf}
\usepackage{epsfig}
\usepackage{pdfsync}    %sumatra inverse search

\usepackage{xspace}
\newtheorem{theorem}{Theorem}[section]
\newtheorem{lemma}[theorem]{Lemma}
\newtheorem{corollary}[theorem]{Corollary}

\newtheorem{proposition}[theorem]{Proposition}

\theoremstyle{definition}
\newtheorem{definition}[theorem]{Definition}

\newtheorem{example}[theorem]{Example}

\newcommand{\nc}{\newcommand}
\newcommand{\delete}[1]{}

\delete{

}

\def\bc{\begin{center}}
	\def\ec{\end{center}}

\nc{\tred}[1]{\textcolor{red}{#1}}
\nc{\tblue}[1]{\textcolor{blue}{#1}} \nc{\tgreen}[1]{\textcolor{green}{#1}} \nc{\tpurple}[1]{\textcolor{purple}{#1}} \nc{\btred}[1]{\textcolor{red}{\bf #1}} \nc{\btblue}[1]{\textcolor{blue}{\bf #1}} \nc{\btgreen}[1]{\textcolor{green}{\bf #1}} \nc{\btpurple}[1]{\textcolor{purple}{\bf #1}}

\newcommand{\efootnote}[1]{}
\nc{\mlabel}[1]{\label{#1}}  % Use this to suppress names
\nc{\mcite}[1]{\cite{#1}}  % Use this to suppress names
\nc{\mref}[1]{\ref{#1}}  % Use this to suppress names
\nc{\meqref}[1]{\eqref{#1}}  % Use this to suppress names
\nc{\mbibitem}[1]{\bibitem{#1}} % Use this to show number
\delete{
	\nc{\mlabel}[1]{\label{#1}  % Use the next two lines to show names
		{\hfill \hspace{1cm}{\bf{{\ }\hfill(#1)}}}}
	\nc{\mcite}[1]{\cite{#1}{{\bf{{\ }(#1)}}}}  % Use this lines to show names
	\nc{\mref}[1]{\ref{#1}{{\bf{{\ }(#1)}}}}  % Use this lines to show names
	\nc{\meqref}[1]{\eqref{#1}{{\bf{{\ }(#1)}}}}  % Use this lines to show names
	\nc{\mbibitem}[1]{\bibitem[\bf #1]{#1}} % Use this to show name
}

\renewcommand\geq{\geqslant}
\renewcommand\leq{\leqslant}

\renewcommand\bar[1]{\overline{#1}}
\nc{\name}[1]{{\bf #1}}

\nc{\tforall}{\quad \text{ for all }}

\nc{\mre}{\mathrm{Re}\,}

\nc{\mim}{\mathrm{im}\,}
\nc{\nz}{\varepsilon}
\nc{\Id}{\mathrm{Id}}

\nc{\DO}{\mathrm{DO}}
\nc{\IDO}{\mathrm{IDO}}

\nc{\mnoindent}{\smallskip\noindent}

\nc{\bin}[2]{ (_{\stackrel{\scs{#1}}{\scs{#2}}})}  %binomial coeff
\nc{\binc}[2]{ \left (\!\! \begin{array}{c} \scs{#1}\\
		\scs{#2} \end{array}\!\! \right )}  %binomial coeff
\nc{\bincc}[2]{  \left ( {\scs{#1} \atop
		\vspace{-1cm}\scs{#2}} \right )}  %binomial coeff
\nc{\bs}{\bar{S}} \nc{\cosum}{\sqsubset} \nc{\la}{\longrightarrow} \nc{\rar}{\rightarrow} \nc{\dar}{\downarrow} \nc{\dprod}{**} \nc{\dap}[1]{\downarrow \rlap{$\scriptstyle{#1}$}} \nc{\md}[1]{\bar{#1}} \nc{\uap}[1]{\uparrow \rlap{$\scriptstyle{#1}$}} \nc{\defeq}{\stackrel{\rm def}{=}} \nc{\disp}[1]{\displaystyle{#1}} \nc{\dotcup}{\ \displaystyle{\bigcup^\bullet}\ } \nc{\gzeta}{\bar{\zeta}} \nc{\hcm}{\ \hat{,}\ } \nc{\hts}{\hat{\otimes}} \nc{\barot}{{\otimes}} \nc{\free}[1]{\bar{#1}} \nc{\uni}[1]{\tilde{#1}} \nc{\hcirc}{\hat{\circ}} \nc{\leng}{\ell} \nc{\lleft}{[} \nc{\lright}{]} \nc{\lc}{\lfloor} \nc{\rc}{\rfloor}
\nc{\lb}{[} %left bracket
\nc{\rb}{]} %right bracket
\nc{\curlyl}{\left \{ \begin{array}{c} {} \\ {} \end{array}
	\right.  \!\!\!\!\!\!\!}
\nc{\curlyr}{ \!\!\!\!\!\!\!
	\left. \begin{array}{c} {} \\ {} \end{array}
	\right \} }
\nc{\longmid}{\left | \begin{array}{c} {} \\ {} \end{array}
	\right. \!\!\!\!\!\!\!}
\nc{\onetree}{\bullet} \nc{\ora}[1]{\stackrel{#1}{\rar}}
\nc{\ola}[1]{\stackrel{#1}{\la}}%${\Bbb Z}$
\nc{\ot}{\otimes} \nc{\mot}{{{\boxtimes\,}}} \nc{\otm}{\overline{\boxtimes}} \nc{\sprod}{\bullet} \nc{\scs}[1]{\scriptstyle{#1}} \nc{\mrm}[1]{{\rm #1}} \nc{\msum}{\sum\limits}
\nc{\margin}[1]{\marginpar{\rm #1}}   %{\rm #1}}
\nc{\dirlim}{\displaystyle{\lim_{\longrightarrow}}\,} \nc{\invlim}{\displaystyle{\lim_{\longleftarrow}}\,} \nc{\mvp}{\vspace{0.3cm}} \nc{\tk}{^{(k)}} \nc{\tp}{^\prime} \nc{\ttp}{^{\prime\prime}} \nc{\svp}{\vspace{2cm}} \nc{\vp}{\vspace{8cm}} \nc{\proofbegin}{\noindent{\bf Proof: }}
\nc{\proofend}{$\blacksquare$ \vspace{0.3cm}}
\nc{\modg}[1]{\!<\!\!{#1}\!\!>}
\nc{\intg}[1]{F_C(#1)} \nc{\lmodg}{\!<\!\!} \nc{\rmodg}{\!\!>\!} \nc{\cpi}{\widehat{\Pi}}
\nc{\sha}{{\mbox{\cyr X}}}  %used to be \cyr
\nc{\shap}{{\mbox{\cyrs X}}} %sha as product
\nc{\shpr}{\diamond}    %Shuffle product
\nc{\shp}{\ast} \nc{\shplus}{\shpr^+}
\nc{\shprc}{\shpr_c}    %Cartier's product
\nc{\msh}{\ast} \nc{\zprod}{m_0} \nc{\oprod}{m_1} \nc{\vep}{\varepsilon} \nc{\labs}{\mid\!} \nc{\rabs}{\!\mid}
\nc{\astarrow}{\overset{\raisebox{-3pt}{$\ast$}}{\rightarrow}}
\nc{\sqsym}{Stirling quasisymmetric function\xspace}
\nc{\sqsyms}{Stirling quasisymmetric functions\xspace}

\nc{\EEsym}{\mathbb{E}sym}
\nc{\Sym}{\mrm{Sym}}
\nc{\NSym}{\mrm{NSym}}
\nc{\QSym}{\mrm{QSym}}
\nc{\RQSym}{\mrm{RQSym}}
\nc{\RenQSym}{\mrm{WCQSym}}	%renormalized quasisymmetric
\nc{\DQSym}{\mrm{DQSym}}
\nc{\WDQSym}{\mrm{WDQSym}}
\nc{\DLQSym}{\mrm{DLQSym}}
\nc{\ZQSym}{\mrm{ZQSym}}
\nc{\Ensym}{\mrm{ENSym}}
\nc{\Wcsym}{\mrm{WCSym}}
\nc{\LWQSym}{\mrm{LWQSym}}
\nc{\LWCQSym}{\mrm{\mathrm{LWQSym}}}
\nc{\Wcqsym}{\mrm{QSym}_{\widetilde{\mathbb{N}}}}%{\mrm{\mathcal{W}QSym}}
\nc{\Syms}{symmetric functions\xspace}
\nc{\eqsym}{extended quasisymmetric function\xspace}
\nc{\eqsyms}{extended quasisymmetric functions\xspace}
\nc{\Eqsyms}{Extended Quasisymmetric functions\xspace}
\nc{\Esyms}{Extended symmetric functions\xspace}
\nc{\sgqsym}{quasisymmetric function with semigroup exponents\xspace}
\nc{\sgqsyms}{quasisymmetric functions with semigroup exponents\xspace}
\nc{\Sgqsyms}{Quasisymmetric functions with semigroup exponents\xspace}
\nc{\SGQSYM}{\mrm{SGQSYM}}
\nc{\emzv}{extended multiple zeta value}
\nc{\emzvs}{extended multiple zeta values}
\nc{\sgfps}{formal power series with semigroup exponent\xspace}
\nc{\NSymg}{\mathrm{NSym}_\gp}
\nc{\zqsym}{zeta-quasisymmetric }
\nc{\gslwqsym}{Stirling left weak quasisymmetric function\xspace}
\nc{\gslwqsyms}{Stirling left weak quasisymmetric functions\xspace}
\nc{\ulwb}{upper-left weak bicomposition\xspace}
\nc{\ulwbs}{upper-left weak bicompositions\xspace}

\nc{\parr}{\rm Par}
\nc{\wpar}{\rm WPar}
\nc{\wcomp}{\large{\VDash}}
\nc{\Ker}{\ker}

\nc{\dth}{d} \nc{\mmbox}[1]{\mbox{\ #1\ }} \nc{\fp}{\mrm{FP}} \nc{\rchar}{\mrm{char}} \nc{\Fil}{\mrm{Fil}} \nc{\Mor}{Mor\xspace} \nc{\gmzvs}{gMZV\xspace} \nc{\gmzv}{gMZV\xspace} \nc{\mzv}{MZV\xspace} \nc{\mzvs}{MZVs\xspace}
\nc{\MZV}{\mathrm{MZV}}
\nc{\Hom}{\mrm{Hom}} \nc{\id}{\mrm{id}} \nc{\im}{\mrm{im}} \nc{\incl}{\mrm{incl}}  \nc{\mchar}{\rm char}

\nc{\Alg}{\mathbf{Alg}} \nc{\Bax}{\mathbf{Bax}} \nc{\bff}{\mathbf f} \nc{\bfk}{{\bf k}} \nc{\bfone}{{\bf 1}} \nc{\bfx}{\mathbf x} \nc{\bfy}{\mathbf y}
\nc{\base}[1]{\bfone^{\otimes ({#1}+1)}} %{{a_{#1}}}
\nc{\Cat}{\mathbf{Cat}} \delete{}
\nc{\detail}{\marginpar{\bf More detail}
	\noindent{\bf Need more detail!}
	\svp}
\nc{\Int}{\mathbf{Int}} \nc{\Mon}{\mathbf{Mon}}
\nc{\rbtm}{{shuffle }} \nc{\rbto}{{Rota-Baxter }} \nc{\remarks}{\noindent{\bf Remarks: }} \nc{\Rings}{\mathbf{Rings}} \nc{\Sets}{\mathbf{Sets}}
\nc{\balpha}{\mathbf{\alpha}}

\nc{\BA}{{\mathbb A}} \nc{\CC}{{\mathbb C}} \nc{\DD}{{\mathbb D}} \nc{\EE}{{\mathbb E}} \nc{\FF}{{\mathbb F}} \nc{\GG}{{\mathbb G}} \nc{\HH}{{\mathbb H}} \nc{\LL}{{\mathbb L}} \nc{\NN}{{\mathbb N}} \nc{\KK}{{\mathbb K}} \nc{\PP}{{\mathbb P}} \nc{\QQ}{{\mathbb Q}} \nc{\RR}{{\mathbb R}} \nc{\TT}{{\mathbb T}} \nc{\VV}{{\mathbb V}} \nc{\ZZ}{{\mathbb Z}}

\nc{\cala}{{\mathcal A}} \nc{\calc}{{\mathcal C}} \nc{\cald}{{\mathcal D}} \nc{\cale}{{\mathcal E}} \nc{\calf}{{\mathcal F}} \nc{\calg}{{\mathcal G}} \nc{\calh}{{\mathcal H}} \nc{\cali}{{\mathcal I}} \nc{\call}{{\mathcal L}} \nc{\calm}{{\mathcal M}} \nc{\caln}{{\mathcal N}} \nc{\calo}{{\mathcal O}} \nc{\calp}{{\mathcal P}} \nc{\calr}{{\mathcal R}} \nc{\cals}{{\mathcal S}} \nc{\calt}{{\mathcal T}} \nc{\calw}{{\mathcal W}} \nc{\calk}{{\mathcal K}} \nc{\calx}{{\mathcal X}}
\nc{\calz}{{\mathcal Z}}

\nc{\fraka}{{\mathfrak a}} \nc{\frakA}{{\mathfrak A}} \nc{\frakb}{{\mathfrak b}} \nc{\frakB}{{\mathfrak B}}
\nc{\frakc}{{\mathfrak c}}  \nc{\frakD}{{\mathfrak D}}
\nc{\frakH}{{\mathfrak H}}
\nc{\frakh}{{\mathfrak h}} \nc{\frakM}{{\mathfrak M}}
\nc{\frakO}{{\mathfrak O}}
\nc{\frakE}{{\mathfrak E}}
\nc{\bfrakM}{\overline{\frakM}} \nc{\frakm}{{\mathfrak m}} \nc{\frakP}{{\mathfrak P}} \nc{\frakN}{{\mathfrak N}} \nc{\frakp}{{\mathfrak p}} \nc{\frakS}{{\mathfrak S}}
\nc{\frakk}{{\mathfrak k}}
\nc{\frakx}{{\mathfrak x}}
\nc{\frakl}{{\mathfrak l}} \nc{\ox}{\bar{\frakx}} \nc{\frakX}{{\mathfrak X}} \nc{\fraky}{{\mathfrak y}} \nc\dop{\delta}
\nc{\Reduce}{{\rm Red}}

\font\cyr=wncyr10 \font\cyrs=wncyr7
\nc{\redt}[1]{\textcolor{red}{#1}}
\nc{\li}[1]{\textcolor{red}{#1}}
\nc{\lir}[1]{\textcolor{red}{Li:#1}}
\nc{\ap}[1]{\textcolor{blue}{#1}}
\nc{\apr}[1]{\textcolor{blue}{AP:#1}}
\nc{\mrep}{\mathrm{Id}}
\nc{\Fix}{\mathrm{Fix}}

\nc{\wvec}[2]{{\scriptsize{\Big [ \!\!\begin{array}{c} #1 \\ #2 \end{array} \!\! \Big ]}}}

\nc{\bwvec}[2]{\Big(\wvec{#1}{#2}\Big)}

\nc{\jwvec}[2]{{\scriptsize{\Big [ \!\!\begin{array}{cccccccccccccc} #1 \\ #2 \end{array} \!\! \Big ]}}}
\nc{\bjwvec}[2]{\Big(\jwvec{#1}{#2}\Big)}

\nc{\diffc}{difference operator\xspace}
\nc{\Diffc}{Difference operator\xspace}
\nc{\diffcs}{difference operators\xspace}

\nc{\RBO}{\mathrm{RBO}}

\nc{\DW}{\mathrm{DW}}

\begin{document}
	
\title[Difference operators on lattices]{Difference operators on lattices}

\author{Aiping Gan}
\address{School of Mathematics and Statistics,
Jiangxi Normal University, Nanchang, Jiangxi 330022, P.R. China}
\email{ganaiping78@163.com}

\author{Li Guo}
\address{Department of Mathematics and Computer Science, Rutgers University, Newark, NJ 07102, USA}
\email{liguo@rutgers.edu}
	
\hyphenpenalty=8000
	
\date{\today}
	
\begin{abstract}
A differential operator of weight $\lambda$ is the algebraic abstraction of the difference quotient $d_\lambda(f)(x):=\big(f(x+\lambda)-f(x)\big)/\lambda$, including both the derivation as   $\lambda$ approaches to $0$  
and the difference operator when $\lambda=1$. Correspondingly, differential algebra of weight $\lambda$ extends the well-established theories of differential algebra and difference algebra.	
%Derivations have been studied with motivation from information science. 
In this paper, we initiate the study of differential operators with weights, in particular difference operators,  on lattices. We show that differential operators of weight $-1$ on a lattice coincide with differential operators, while differential operators are special cases of \diffcs. Distributivity of a lattice is characterized by the existence of certain \diffcs. 
Furthermore, we characterize and enumerate \diffcs on finite chains and finite  quasi-antichains. 
\end{abstract}
	
\subjclass[2010]{
%	03G25,
%	06F35,
06B20,	%varieties of lattices
06D05,	%Structure and preresentation of distributive lattices
12H10,	%differene algebra
13N15,	%Derivations and commutative rings
05A15,	%exact enumeration problems
39A70	%difference operators
}
	
\keywords{lattice; derivation; difference lattice; distributive lattice; Catalan number; chain; quasi-antichain\\
Corresponding author: Li Guo (liguo@rutgers.edu)}
	
\maketitle

\vspace{-1cm}
\tableofcontents

\vspace{-1cm}

\hyphenpenalty=8000 \setcounter{section}{0}
	
%========================================================================
	
\allowdisplaybreaks

\section {Introduction}	
\mlabel{sec:intr}

This paper studies the notion of differential operators with weights, in particular difference operators, on lattices. We show that these operators are closely related to common properties of lattices and provide their classifications on chains and quasi-antichains. 

\subsection{Derivations and lattices}
The derivation, or differential operator, and integral operator are fundamental in analysis and its broad applications. As an abstraction of the derivation, the notion of a differential algebra was introduced in the 1930's by Ritt~\mcite{Ri}, to be a field $A$ carrying a linear operator $d$ satisfying an abstract of the Leibniz rule of the derivation:
$$d(uv)=d(u) v+ud(v) \quad \text{ for all } u, v\in A.$$
Thus $d$ is still called a derivation or a differential operator. 
The theory of differential algebra for fields and more generally for commutative algebras has since been expanded into a major area of mathematical research including differential Galois theory, differential algebraic geometry and differential algebraic groups, with broad applications in arithmetic geometry, logic and computational algebra~\mcite{CGKS,Ko,SP}.

Furthermore, in connection with combinatorics, differential structures were found on heap ordered trees~\mcite{GrL} and on decorated rooted trees~\mcite{GK3}.
The operad of differential associative algebras was studied in~\mcite{Lod}. 
Moreover, derivations on commutative algebras gave rise to the structure of Poisson algebras, perm algebras and Novikov algebras~\mcite{BV1,CL,GD}. 

As another major algebraic structure with broad applications, lattice theory~\mcite{Bir,Bly,Davey} has been developed in close connection with universal algebra~\mcite{BS,Gr2}. The notion of derivations on lattices was introduced by Szasz~\mcite{Sz} and further developed  by Ferrari \mcite{fer}, among others. In their language, a derivation on a lattice $(L,\vee,\wedge)$ is a map $d:L\to L$ satisfying
\begin{equation}
d(x\vee y)=d(x)\vee d(y), \quad d(x\wedge y)=(d(x)\wedge y)\vee (x\wedge d(y)) \tforall x, y\in L.
\mlabel{eq:00}
\end{equation}

The notion of derivations  without requiring the first condition was 
investigated by Xin and coauthors
\mcite{Xin1,Xin2}. This study was continued in~\mcite{GG} from the viewpoint of universal algebra, and integral operators on lattice was studied in \mcite{GGW}.
There are also studies on generalizations of derivations on lattices, such as generalized derivations \mcite{als2}, higher derivations \mcite{Ce1},
$n$-derivations and $(n, m)$-derivations \mcite{Ce2}, $f$-derivations~\mcite{CO}.

\subsection{Difference operators and lattices}
As a discrete analogy of the derivation and a fundamental notion in numerical analysis, the difference operator is defined by sending a function $f(x)$ to the difference $f(x+1)-f(x)$. More generally, for any nonzero real number $\lambda$, define the difference quotient operator $$d_{\lambda}(f)(x):=\frac{f(x+\lambda)-f(x)}{\lambda}.$$
Then $d_\lambda$ satisfies the operator identity
\begin{equation}\mlabel{eq:diffw}
d_\lambda(fg)=d_\lambda(f)g+fd_\lambda(g)+\lambda d_\lambda(f)d_\lambda(g)  \quad \tforall \text{functions} ~f, g.
\end{equation}

Motivated by this property, the notion of a {\bf differential operator of weight $\lambda$} was introduced in~\mcite{GK3}, as a linear operator $d:R\to R$ on any associative algebra $R$ that satisfies Eq.~\meqref{eq:diffw1}:
\begin{equation}\mlabel{eq:diffw1}
	d(fg)=d(f)g+fd(g)+\lambda d(f)d(g)  \quad \tforall f, g\in R.
\end{equation}
The special case when $\lambda=1$ generalizes the usual difference operator;  
while when $\lambda=-1$, we obtain the backward difference operator 
$$ d_{-1}(f)(x):=f(x-1)-f(x)$$
up to a sign. Of course, when $\lambda$ is taken to be zero, we recover the usual derivation. 

These operators have been considered for other algebraic structures with a linear structure, such as Lie algebras~\mcite{LGG}. For an algebraic structure like lattice that does not have a linear structure, the operators identities in the three special cases, when $\lambda=0, 1, -1$, can still be defined, motivating us to give the following definitions. See~\mcite{GLS} for the notion of differential operators on groups.

\begin{definition}
	Let $L$ be a lattice. A map $d:L\to L$ is called 
\begin{enumerate}
\item a \name {derivation}~\mcite{Xin1} or a \name {differential operator (of weight $0$)} if $d$ satisfies
\begin{equation}
	d(x\wedge y)=(d(x)\wedge y)\vee (x\wedge d(y)) \quad \tforall ~ x, y\in L;
	\mlabel{eq:der}
\end{equation}
\item a \name {differential operator of weight $1$}  or a \name {difference operator} if $d$ satisfies 
	\begin{equation}
	d(x\wedge y)=(d(x)\wedge y)\vee (x\wedge d(y))\vee (d(x)\wedge d(y)) \tforall x, y\in L;
	\mlabel{eq:301}
\end{equation}
\item a \name{differential operator of weight $-1$} if $d$ satisfies 
\begin{equation}
	d(x\wedge y)\vee (d(x)\wedge d(y))=(d(x)\wedge y)\vee (x\wedge d(y)) \tforall x, y\in L.
	\mlabel{eq:30}
\end{equation}
\end{enumerate}
\mlabel{d:31}
\mlabel{d:311}
\end{definition}

Adapting the classical terminology, we  call a map $f:L\to L$ an \name{operator} even though there is no linearity imposed. 

As noted above, so far the related study in lattice theory has been focused on the derivations, that is, differential operator of weight zero. Due to the importance of differential operators with weights and in particular difference operators, we initiate the study of these operators on lattices in this paper. 
An unexpected property of these operators on lattices is that 
differential operators of weight $-1$ coincide with differential operators (of weight $0$). Furthermore, differential operators are precisely the decreasing difference operators. Enumerations of difference operators are linked to Catalan numbers. 

\subsection{Layout of the paper} 
The paper is organized as follows. 

In Section~\mref{sec:gen}, we first show that differential operators of weight $-1$  coincide with derivations on a given lattice (Theorem \mref{theo:30}).
We then focus on \diffcs in the rest of the paper. There are plenty of supply of \diffcs. Derivations are just decreasing \diffcs  (Proposition \mref{pro:301}). Further, there are several families of \diffcs and new \diffcs from existing ones (Propositions \mref{por:111} and \mref{por:113}).
Also, distributive lattices are characterized by the existence of a family of \diffcs (Theorem \mref{the:001}). 

In Section~\mref{sec:chain}, we classify and enumerate \diffcs on finite chains. The numbers are related to Catalan numbers (Theorem \mref{the:400}).
 
In Section~\mref{sec:qua}, we classify and enumerate \diffcs on quasi-antichains by dividing the discussion into several cases.

\section{Differential operators with weights and difference operators}
\mlabel{sec:gen}
In this section, we study general properties of differential operators of weight $-1$ or $1$ on a lattice. We show that there is a large supply of \diffcs some of which characterize the distributivity of lattices. 

\subsection{Differential operators of weight $-1$ and derivations}
\mlabel{sec:dow}

Let $\DW_\lambda(L)$ denote the set of differential operators of weight $\lambda\in \{0,1,-1\}$ on a lattice $L$. The set of differential operators of weight zero on $L$ is also denoted by $\DO(L)$ in~\mcite{GG}. 
We recall the following result for later applications.
\begin{lemma}
	\name{\mcite{Xin1}} Let $L$ be a lattice and $d\in \DO(L)$. Then the following statements hold.
	\begin{enumerate}
		\item $d$ is decreasing, i.e.,
		$d(x)\leq x$ for all $x\in L$. In particular, $d(0)=0$ if $L$ has a bottom element $0$.
		\mlabel{it:2011}
		\item $d(x)\wedge d(y)\leq x\wedge d(y)\leq d(x\wedge y)$ for all $x, y\in L$.
		\mlabel{it:2012}
		\item If $x\in L$ such that $x\leq d(u)$ for some $u\in L$, then $d(x)=x$.
		\mlabel{it:2013}
		\item If $L$ has a top element $1$ and $d(1)=1$, then $d=\mrep_{L}$, the identity map on $L$.
		\mlabel{it:2014}
		\item $d$ is  idempotent, i.e, $d^{2}=d$.
		\mlabel{it:2015}
	\end{enumerate}
	\mlabel{pro:201}
\end{lemma}

Next, we will show that
$\DW_{-1}(L)$ coincides with $\DO(L)$. The case when of $\DW_1(L)$ will be studied in the later sections. 

\begin{lemma}
 Let $L$ be a  lattice and $d\in \DW_{-1}(L)$. Then $d$ is decreasing.
\mlabel{lm:301}
\end{lemma}
\begin{proof}
Assume that $L$ is a lattice and $d\in \DW_{-1}(L)$. Let $x\in L$. Then
$$d(x)=d(x\wedge x)\vee (d(x)\wedge d(x))=(d(x)\wedge x)\vee (x\wedge d(x))=d(x)\wedge x$$ by  Eq.\meqref{eq:30}, and so  $d(x)\leq x$. Thus $d$ is decreasing.
\end{proof}

\begin{theorem}
Let  $L$ be a  lattice. Then $\DO(L)= \DW_{-1}(L)$.
\mlabel{theo:30}
\end{theorem}
\begin{proof}
%We first prove  $\DO(L)\subseteq \DW_{-1}(L)$.
Let  $d\in \DO(L)$. Then for all $x, y\in L$,   Lemma \mref{pro:201} gives
$d(x)\wedge d(y)\leq d(x)\wedge y\leq d(x\wedge y),$
and so
	$d(x\wedge y)\vee (d(x)\wedge d(y))=d(x\wedge y)=(d(x)\wedge y)\vee (x\wedge d(y))$
by Eq.\meqref{eq:der}. Thus $d\in\DW_{-1}(L)$, and hence
	$\DO(L)\subseteq \DW_{-1}(L)$.
	
To prove the opposite inclusion, let  $d\in \DW_{-1}(L)$ and  $x, y\in L$. Then we have
$$d(x\wedge y)=d((x\wedge y)\wedge y)\vee (d(x\wedge y)\wedge d(y))= (d(x\wedge y)\wedge y)\vee ((x\wedge y)\wedge d(y))$$
by  Eq.\meqref{eq:30}, which implies that
$d(x\wedge y)
=d(x\wedge y)\vee (x\wedge d(y))$
since  $d(y)\leq y$ and $d(x\wedge y)\leq x\wedge y\leq y$ by Lemma \mref{lm:301}, and so
 $x\wedge d(y)\leq d(x\wedge y)$.   Since $d(x)\leq  x$ by Lemma \mref{lm:301}, we get $d(x)\wedge d(y)\leq x\wedge d(y)\leq d(x\wedge y)$.Thus
$$d(x\wedge y)=d(x\wedge y)\vee (d(x)\wedge d(y))=(d(x)\wedge y)\vee (x\wedge d(y))$$ 
by Eq.\meqref{eq:30}, and hence
 $d\in\DO(L)$. This shows that
$ \DW_{-1}(L)\subseteq \DO(L)$, proving the theorem. 
\end{proof}

\subsection{Properties and first examples of \diffcs}
\mlabel{sec:dow1}

In this subsection, we give basic properties and  examples of \diffcs on  lattices. 
We first make a useful simple observation.

\begin{lemma}
	For any operator $d$ on
	a lattice $L$, the difference operator relation in Eq.~\meqref{eq:301} holds when $x=y$. 
	\mlabel{lem:x=y}
\end{lemma}
\begin{proof}
The lemma is clear since both sides of
Eq.~\meqref{eq:301}
is $d(x)$ when $x=y$. 
\end{proof}

\begin{example}
	\begin{enumerate}
		\item Let $L$ be a lattice and $a\in L$. Define $\textbf{C}_{(a)}: L\to L$ by $\textbf{C}_{(a)}(x):=a$ for all $x\in L$. It is easy to check that
		$\textbf{C}_{(a)}\in \DW_{1}(L)$.
		$\textbf{C}_{(a)}$ is called  the \name{constant   operator  with value $a$}.
		\mlabel{it:3011e}
		\item Let $L$ be a lattice with a top element $1$ and $|L|\geq 2$.  For each $a\in L$, define 
		$\tau^{(a)}: L\to L$ by 
		$$
		\tau^{(a)}(x):=
		\begin{cases}
			a,  & \textrm{if}~ x=1; \\
			1,  & \textrm{otherwise}.
		\end{cases}
		$$	
	Then	$\tau^{(a)}\in \DW_{1}(L)$
		(a detailed proof is given in Proposition \mref{por:111}). 
		\mlabel{it:3012e}
	\end{enumerate}
	\mlabel{exa:301}
\end{example}

Here are some basic properties. 
\begin{proposition}	\mlabel{pro:3000}
Let $L$ be a lattice with a top element $1$ and $d\in \DW_{1}(L)$. Then the following statements hold.
	\begin{enumerate}
		\item If $x\in L$ such that  $x\leq d(1)$, then
		$x\leq d(x)$. In particular,
			$d(1)\leq d^{2}(1)$.
		\mlabel{it:3001}
		
		\item If $x, y\in L$ such that $x\leq y\leq d(1)$, then $d(x)\leq d(y)$.
	\mlabel{it:3002}
		\item If  $d(1)=1$, then $d$ is  increasing (i.e.,
		$x\leq d(x)$ for all $x\in L$) and isotone. 
		\mlabel{it:3003}
		\item If $x\in L$ and  $ d(1)\leq x$, then
		$d(1)\leq d(x)$.
\mlabel{it:3004}
		 \item If $L$ has a bottom  element $0$ and $x\in L$ such that $x\vee d(x)\leq d(0)$, then  $x\vee d(x)= d(0)$.
		In particular, if $d(0)=1$, then  $x\vee d(x)=1$.
		\mlabel{it:3005}
	 \item 
	 If $d(a)=1$ for some $a\in L\backslash \{1\}$, then $x\vee d(x)=1$ for all  $x\in L$ with $a\leq x$. 	\mlabel{it:3006}
	\end{enumerate}
\end{proposition}
\begin{proof}
Assume that $L$ is a lattice with a top element $1$ and let $d\in \DW_{1}(L)$.
	
	\mnoindent
\mref{it:3001}	If 	 $x\in L$ and  $x\leq d(1)$, then
$x=x\wedge d(1)\leq d(x\wedge 1)=d(x)$
 by Eq. \meqref{eq:301}.

	\mnoindent
\mref{it:3002} If $x, y\in L$ such that $x\leq y\leq d(1)$, then
$x\leq d(x)$ and $y\leq d(y)$ by Item \mref{it:3001}, which implies that 
$d(x)=	d(x\wedge y)=(d(x)\wedge y)\vee (x\wedge d(y))\vee (d(x)\wedge d(y))=d(x)\wedge d(y)$ by 
Eq.\meqref{eq:301}. Thus
 $d(x)\leq d(y)$.
 
 	\mnoindent
 \mref{it:3003} follows immediately from Items
 \mref{it:3001} and \mref{it:3002}.

	\mnoindent
\mref{it:3004}
	If $x\in L$ and $d(1)\leq x$,  then $d(1)=x\wedge d(1)\leq d(x\wedge 1)=d(x)$
by Eq.\meqref{eq:301}.

	\mnoindent
\mref{it:3005}	If 	 $L$ has a bottom  element $0$ and $x\in L$ such that $x\vee d(x)\leq d(0)$,  then $x\leq d(0)$ and $d(x)\leq d(0)$, so
$d(0)=	d(x\wedge 0)=(d(x)\wedge 0)\vee (x\wedge d(0))\vee (d(x)\wedge d(0)) =x\vee d(x)$ by 
Eq.\meqref{eq:301}.

	\mnoindent
\mref{it:3006}		If $d(a)=1$ for some $a\in L\backslash \{1\}$, then for all	 $x\in L$ with $a\leq x$, we have
$1=d(a)=	d(x\wedge a)=(d(x)\wedge a)\vee (x\wedge d(a))\vee (d(x)\wedge d(a))=x\vee d(x)$ by
Eq.\meqref{eq:301}.
\end{proof}

There is a close relationship between differential operators and \diffcs on lattices.
\begin{proposition}
	 Let  $d$ be an operator on a lattice $L$. Then the following statements hold.
\begin{enumerate}
	\item  $\DO(L)=\{d\in \DW_{1}(L) ~| ~d$ is decreasing$ \}$.
		\mlabel{it:3011}
	\item If $d$ is increasing, then $d\in \DW_{1}(L)$ if and only if
$d$ is a $\wedge$-homomorphism, i.e.,	
	 $d(x\wedge y)=d(x)\wedge d(y)$ for all $x, y\in L$.
		\mlabel{it:3012}
\end{enumerate}	 
	 	\mlabel{pro:301}
\end{proposition}
\begin{proof}
	Assume that $d$ is an operator on a lattice $L$.
	
	\mnoindent
\mref{it:3011}	If  $d\in \DO(L)$, then
Lemma \mref{pro:201} gives that
 $d$ is decreasing
 and $d(x)\wedge d(y)\leq x\wedge  d( y)$ for all $x, y\in L$, which implies by Eq.\meqref{eq:der} that
	$$d(x\wedge y)=(d(x)\wedge  y)\vee (x\wedge d(y))=(d(x)\wedge y)\vee (x\wedge d(y))\vee (d(x)\wedge d(y)).$$ Thus $d\in\DW_{1}(L)$, and hence
	 $\DO(L)\subseteq\{d\in \DW_{1}(L) ~|~ d$ is decreasing$ \}.$

Conversely, if $d\in \DW_{1}(L)$ and $d$ is decreasing, then for all $ x, y\in L$, we have $d(x)\leq x$, and so
$d(x)\wedge d(y)\leq x \wedge d(y) $, which implies by Eq.\meqref{eq:301} that 
$$	d(x\wedge y)=(d(x)\wedge y)\vee (x\wedge d(y))\vee (d(x)\wedge d(y))=(d(x)\wedge y)\vee (x\wedge d(y)),$$
and thus $d\in \DO(L)$. Therefore
 $ \DO(L)=\{d\in \DW_{1}(L) ~|~ d$ is decreasing$ \}$. 

	\mnoindent
\mref{it:3012}
Assume that $d$ is increasing, and let $x, y\in L$. Then $x\leq d(x)$ and $y\leq d(y)$. So by Eq.\meqref{eq:301} we have
\begin{eqnarray*}
	d\in \DW_{1}(L) &\Leftrightarrow&
	d(x\wedge y)=(d(x)\wedge y)\vee (x\wedge d(y))\vee (d(x)\wedge d(y))	\\	&\Leftrightarrow&
	d(x\wedge y)=d(x)\wedge d(y). \hspace{5cm} \qedhere
\end{eqnarray*}	
\end{proof}

Proposition \mref{pro:301} states that every differential operator on a lattice is a \diffc.
The converse does not hold, as shown in the following examples.

\begin{example}  Let $L$ be a lattice and $a\in L$. 
\begin{enumerate}
\item If $L$ has a bottom element $0$ and
$a\neq 0$, then the constant operator $\textbf{C}_{(a)}$ with value $a$ defined in Example~\mref{exa:301} \mref{it:3011e} is in  $\DW_1(L)$ but not in
$\DO(L)$ by Lemma \mref{pro:201} \mref{it:2011}, since $\textbf{C}_{(a)}(0)=a\neq 0$. 
\mlabel{it:3011b}

\item If $L$ has a top element $1$ and $|L|\geq 2$, then the operator
$\tau^{(a)}$ (in Example~\mref{exa:301}\mref{it:3012e})
is in $\DW_{1}(L)$ but not in $\DO(L)$  by Lemma \mref{pro:201} \mref{it:2013}, since $\tau^{(a)}$ is not decreasing. 	
	\mlabel{it:3012b}
\end{enumerate}
 \mlabel{exa:301b}
\end{example}

The following properties suggests that \diffcs depend very much on their actions on the tops and bottoms. Further results in this direction can be found in Section~\mref{sec:qua}. 

\begin{proposition}
	Let $(L, \vee, \wedge, 0, 1)$ be a bounded lattice  and $d$ be an operator on $L$ such that $d(0)=d(1)=1$. Then $d\in \DW_{1}(L)$ if and only if
		$d
	=\textbf{C}_{(1)}$, where $\textbf{C}_{(1)}$ is the constant operator defined in Example \mref{exa:301} \mref{it:3011e}. 
	\mlabel{pro:001}
\end{proposition}
\begin{proof}
Let $L$ and $d$ be as given in the proposition. If $d\in \DW_{1}(L)$, then for all $x\in L$, since 
$0\leq x\leq 1=d(1)$, we have  
$1=d(0)\leq d(x)$ by Proposition \mref{pro:3000} \mref{it:3002}, and
so $d(x)=1$. Thus  $d=\textbf{C}_{(1)}$.

Conversely, if	$d=\textbf{C}_{(1)}$, then  $d\in \DW_{1}(L)$ by Example \mref{exa:301} \mref{it:3011e}.	
\end{proof}

\begin{proposition}
Let $(L, \vee, \wedge, 0, 1)$ be a bounded chain  and $d$ be an operator on $L$ such that $d(0)=1$. Then $d\in \DW_{1}(L)$ if and only if
 $d=\tau^{(a)}$ for some $ a\in L$,
 where $\tau^{(a)}$ is defined in Example \mref{exa:301} \mref{it:3012e}.
	\mlabel{pro:100}
\end{proposition}
\begin{proof} 
Let $L$ and $d$ be as given in the proposition.	
If $d=\tau^{(a)}$ for some $a\in L$, then 	$d\in \DW_{1}(L)$ by Example \mref{exa:301} \mref{it:3012e}.
	
Conversely, if $d\in \DW_{1}(L)$, then  Proposition \mref{pro:3000} \mref{it:3004} gives
 $x\vee d(x)=1$ for all $x\in L$, which implies that $d(y)=1$ for all $y\in L\backslash \{1\}$ since $L$ is a chain.
Thus 	$d=\tau^{(a)}$, where $a=d(1)\in L$.
\end{proof}

Example \mref{exa:301} \mref{it:3012e} also gives examples of \diffcs that are neither isotone nor idempotent. Now we show that these conditions impose strong restrictions on \diffcs. 
   
\begin{proposition}
	Let $L$ be a  lattice and $d\in \DW_{1}(L)$. Then $d$ is isotone if and only if $d$ is a $\wedge$-homomorphism.%  $d(x\wedge y)=d(x)\wedge d(y)$ for all $x, y\in L$.
	\mlabel{pro:302}
\end{proposition}
\begin{proof}
$(\Longleftarrow)$. Assume that  $d$ is a $\wedge$-homomorphism, and let
$x, y\in L$ with $x\leq y$. Then
  $d(x)=d(x\wedge y)=d(x)\wedge d(y)$, and so $d(x)\leq d(y)$. Thus $d$ is isotone. 	
	
$(\Longrightarrow)$.	If  $d\in \DW_{1}(L)$ and $d$ is isotone, then 
	$d(x\wedge y)\leq d(x)\wedge d(y)$
	for all $x, y\in L$.
	Also, we have $d(x)\wedge d(y)\leq d(x\wedge y)$ by Eq.\meqref{eq:301}. Thus   $d(x\wedge y)=d(x)\wedge d(y)$.	
\end{proof}

\begin{corollary}
	Let $L$ be a  lattice and $d\in \DW_{1}(L)$. Then $d: L\to L$ is a lattice homomorphism if and only if
	$d$ is a $\vee$-homomorphism, i.e., $d(x\vee y)=d(x)\vee d(y)$ for all $x, y\in L$.
	\mlabel{cor:301}
\end{corollary}
\begin{proof}
$(\Longrightarrow)$. It follows immedialtely from the fact the a lattice homomorphism is both a $\vee$-homomorphism and a $\wedge$-homomorphism.

$(\Longleftarrow)$.
 Assume that  $d$ is a $\vee$-homomorphism, and let
$x, y\in L$ with $x\leq y$. Then
$d(y)=d(x\vee y)=d(x)\vee d(y)$, and so $d(x)\leq d(y)$. Thus $d$ is isotone, and therefore 	
$d$ is a lattice homomorphism by Proposition \mref{pro:302}. 	
\end{proof}

\begin{corollary}
	Let $L$ be a  chain with a top element $1$ and $d$ be an operator on $L$ such that $d(1)=1$. Then
	$d\in \DW_{1}(L)$ if and only if
	$d$ is increasing and isotone.
	\mlabel{cor:302}
\end{corollary}
\begin{proof}
Let $L$ and $d$ be as given in the corollary. 
If $d\in \DW_{1}(L)$ and $d(1)=1$, then $d$ is increasing and isotone by Proposition \mref{pro:3000} \mref{it:3003}.	

Conversely, if  $d$
is increasing and isotone, then
notice that in a chain $L$,  $d$
is isotone if and only if $d$ is a $\wedge$-homomorphism, we get
$d\in \DW_{1}(L)$ by  Proposition \mref{pro:301} \mref{it:3012}. 	
\end{proof}

\begin{proposition}	\mlabel{pro:303}
	Let $L$ be a  lattice and $d\in \DW_{1}(L)$. Then $d$ is idempotent if and only if $d(d(x)\wedge d(y))=d(x)\wedge d(y)$ for all $x, y\in L$.
\end{proposition}
\begin{proof}
If  $d\in \DW_{1}(L)$ and $d$ is idempotent, then 	for all $x, y\in L$, we have
by Eq.\meqref{eq:301} that	$d(d(x)\wedge d(y))= (d^{2}(x)\wedge d(y))\vee (d(x)\wedge d^{2}(y))\vee (d^{2}(x)\wedge d^{2}(y))=d(x)\wedge d(y)$.	
	
Conversely, if $d(d(x)\wedge d(y))=d(x)\wedge d(y)$ for all $x, y\in L$, then  $d^{2}(x)=d(d(x)\wedge d(x))=d(x)\wedge d(x)=d(x)$, and  so $d$ is idempotent.
\end{proof}

Let $L$ be a lattice and $a\in L$. Define $\psi_{(a)}: L\to L$ by
 $\psi_{(a)}(x)=x\vee a$ for all $x\in L$. Then $\psi_{(a)}$ is usually not in $\DW_{1}(L)$, as shown below. 
 
 \begin{example}
  Let $M_{3}=\{0, b_{1}, b_{2}, b_{3}, 1\}$ be the modular lattice with Hasse diagram as follows
 $$
 
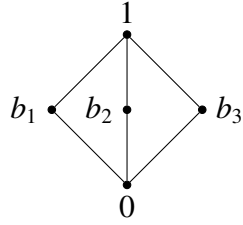
\begin{tikzpicture}
 \tikzstyle{every node}=[draw,circle,fill=black,node distance=1.0cm,
 minimum size=1.0pt, inner sep=1.0pt]
 \node[circle] (1)                        [label=above :   $1$]{};
 \node[circle] (2) [below   of=1]     [label=left : $b_{2}$]{};
 \node[circle] (3) [ left of=2]             [label=left  : $b_{1}$] {};
 \node[circle] (4) [ right   of=2]     [label=right : $b_{3}$]{};
 \node[circle] (5) [below   of=2]     [label=below: $0$] {};
 
 \draw[-] (1) --   (2); \draw[-] (1) --   (3); \draw[-] (1) --   (4);
 \draw[-] (2) --   (5); \draw[-] (3) --   (5);
 \draw[-] (4) --   (5);
\end{tikzpicture}
$$
We have $\psi_{(b_{1})}(b_{2}\wedge b_{3})=b_{1}\neq 1= (\psi_{(b_{1})}(b_{2})\wedge b_{3})\vee (b_{2}\wedge \psi_{(b_{1})}(b_{3}))\vee (\psi_{(b_{1})}(b_{2})\wedge \psi_{(b_{1})}(b_{3}))$, so
 $\psi_{(b_{1})}\not\in \DW_{1}(M_{3})$.
 \mlabel{exa:215}
  \end{example}
 
It is interesting to observe that the distributivity of a lattice can be characterized by $\psi_{(a)}$ being \diffcs. 

\begin{theorem}\mlabel{the:001}
Let $L$ be a lattice. Then $L$ is distributive if and only if  $\psi_{(a)}$ is in $\DW_{1}(L)$ for all $a\in L$. 
\end{theorem}

\begin{proof}
Assume that $L$ is a lattice. It is easy to check that $L$ is distributive if and only if  $\psi_{(a)}$ is a $\wedge$-homomorphism
for any $a\in L$.
Also, since $\psi_{(a)}$ is increasing, it follows immediately from Proposition \ref{pro:301} \ref{it:3012}  that
$\psi_{(a)}\in \DW_{1}(L)$ if and only if $L$ is distributive. 
\end{proof}

\subsection{Derived \diffcs}
\mlabel{su:new}

We now derive new \diffcs from the given ones.

\begin{proposition}
	Let  $L$ be a lattice with a top element $1$ and $d$ be an operator on   $ L$. For a given $u\in L$,  define  an operator $d^{u}$ on $L$ by
\begin{equation} 
	d^{u}(x):=
	\begin{cases}
	u,  & \textrm{if}~ x=1; \\
	d(x),  & \text{otherwise}.
	\end{cases}
\mlabel{eq:du}
\end{equation}
\begin{enumerate} 
	\item If $d\in \DW_1(L)$ and $u\leq d(1)$, then  $d^{u}$ is in $\DW_{1}(L)$. 
	In particular, $\tau^{(a)}$ is in $\DW_{1}(L)$, where $\tau^{(a)}$ is defined in Example \ref{exa:301}. \mlabel{it:111}
\item 	If $d$ satisfies $d(1)\leq u$ and
 $x\wedge u\leq d(x)$ for all $x\in L\backslash \{1\}$, then $d\in \DW_{1}(L)$ if and only if $d^u\in \DW_{1}(L)$. \mlabel{it:112}
 
 \item 	If $d$ satisfies
 $x\leq d(x)$ for all $x\in L\backslash \{1\}$, then $d\in \DW_{1}(L)$ if and only if $d^1\in \DW_{1}(L)$. \mlabel{it:113}
\end{enumerate}
	\mlabel{por:111}
\end{proposition}
\begin{proof}
Let $L, d, u$ and $d^{u}$ be as given in the proposition. 

\noindent	
\mref{it:111} Assume that $d\in \DW_1(L)$ and $u\leq d(1)$.	
To prove  $d^u\in \DW_{1}(L)$,
let $x, y\in L$ which we can assume to be distinct because of Lemma~\mref{lem:x=y}.

If $x\neq 1$ and $y\neq 1$, then 
	$x\wedge y\neq 1$, and so $d^{u}(x)=d(x), d^{u}(y)=d(y)$ and
	$d^{u}(x\wedge y)=d(x\wedge y)$.
	Since $d\in \DW_{1}(L)$, we have	 
	$d^{u}(x\wedge y)=d(x\wedge y)=(d(x)\wedge y)\vee (x\wedge d(y))\vee (d(x)\wedge d(y))=(d^{u}(x)\wedge y)\vee (x\wedge d^{u}(y))\vee (d^{u}(x)\wedge d^{u}(y)).$ 
	
If $ x=1$ or $y=1$ (but not both), say $x=1$ and $ y\neq 1$, then $d^{u}(x)=u$ and $d^{u}(y)=d(y)$.
Since $d^{u}(1)=u\leq d(1)$, we have
	$y\wedge d^{u}(1)\leq y\wedge d(1)\leq d(y)$ by 
	Eq.\meqref{eq:301}, and so
	$d^{u}(x\wedge y)=d^{u}(y)=d(y)=d(y)\vee (y\wedge d^{u}(1))\vee (d^{u}(x)\wedge d(y))=(x\wedge d^{u}(y))\vee(y\wedge d^{u}(x) )\vee  (d^{u}(x)\wedge d^{u}(y)).$
	
Thus we conclude that  $d^{u}$ is in $\DW_{1}(L)$.

\smallskip 

\noindent 
\mref{it:112} Assume that  
$d$ satisfies $d(1)\leq u$ and $x\wedge u\leq d(x)$ for all $x\in L\backslash \{1\}$.	
If $d^u\in \DW_{1}(L)$, then 
 $d=(d^u)^{d(1)}\in  \DW_{1}(L)$ by Item~\mref{it:111}, 
since
$d(1)\leq u=d^u (1)$. 

Conversely, assume that $d\in \DW_{1}(L)$.
To prove  $d^u\in \DW_{1}(L)$, 
let $x, y\in L$ which we can assume to be distinct because of Lemma~\mref{lem:x=y}.
If $x\neq 1$ and $y\neq 1$, then 
similar to the proof of Item \mref{it:111} we have	 
$d^u(x\wedge y)=(d^u(x)\wedge y)\vee (x\wedge d^u(y))\vee (d^u(x)\wedge d^u(y)).$

If $x=1$ or $ y=1$ (but not both), say $x=1$ and $y\neq 1$, then
$d^u(x)= u$, $d^u(y)=d(y)$ and
$u\wedge y\leq d(y)$, which implies that
$d^u(x\wedge y)=d^u(y)=d(y)=(u\wedge y)\vee d(y)=(d^u(x)\wedge y)\vee (x\wedge d^u(y))\vee (d^u(x)\wedge d^u(y)).$
Thus we conclude that  $d^u$ is in $\DW_{1}(L)$.

\smallskip 

\noindent 
\mref{it:113} follows immediately from
Item \ref{it:112} when we take $u=1$.
\end{proof}

Note that, in Proposition \mref{por:111}, if $d\in \DW_1(L)$ but $u\not\leq d(1)$, then  $d^{u}$ needs not be in $\DW_{1}(L)$. For example, let
$L$ be a bounded lattice with $|L|\geq 3$, and  let $d=\textbf{C}_{(0)}$, $u=1$. Then  $d\in \DW_1(L)$ by Example \mref{exa:301} \mref{it:3011e}, but $d^u \not\in \DW_{1}(L)$, since
$d^{u}(1\wedge a)=d^{u}(a)=0\neq a=(d^u(1)\wedge a)\vee (1\wedge d^u(a))\vee (d^u(1)\wedge d^u(a))$ for all $a\in L\backslash \{0, 1\}$.

\begin{proposition}
	Let  $L$ be a chain with the top element $1$ and $d\in\DW_{1}(L)$. 
\begin{enumerate} 	
\item The operator $D$ on $L$ defined by
	$$
	D(x):=
	\begin{cases}
	d(x),  & \textrm{if}~ x\leq d(1); \\
	1,  & \textrm{otherwise}
	\end{cases}
	$$
 is in $\DW_{1}(L)$.
\mlabel{it:dchain1} 
\item 
	The operator $\mathfrak{D}$ on  $L$ defined by
$$
\mathfrak{D}(x):=
\begin{cases}
	d(1),  & \textrm{if}~ x\leq d(1); \\
	d(x),  & \textrm{otherwise}
\end{cases}
$$
is in $\DW_{1}(L)$.
\mlabel{it:dchain2}
\end{enumerate}
	\mlabel{por:113}		
\end{proposition}
\begin{proof}
	Let $L, d, D$ and $\mathfrak{D}$ be as given in the proposition. Let $x, y\in L$.
	
\noindent
\mref{it:dchain1}		
 By Proposition \mref{pro:3000} \mref{it:3001}, we have $x\leq d(x)$ for all $x\leq d(1)$, and so $x\leq D(x)$
	for all $x\in L$. Also, $D$ is isotone
by Proposition \mref{pro:3000} \mref{it:3002}.	
Thus    $D\in \DW_{1}(L)$ by Corollary \mref{cor:302}.

\smallskip

\noindent 
\mref{it:dchain2} 
Since $L$ is a chain,
to prove $\mathfrak{D}\in \DW_{1}(L)$, we assume
$x<y$ (i.e, $x\leq y$ and $x\neq y$)  by Lemma \ref{lem:x=y}. 
If $x<y\leq d(1)$, then $\mathfrak{D}(x)=\mathfrak{D}(y)=d(1)$, and so
$\mathfrak{D}(x\wedge y)=\mathfrak{D}(x)=(\mathfrak{D}(x)\wedge y)\vee (x\wedge \mathfrak{D}(y))\vee (\mathfrak{D}(x)\wedge \mathfrak{D}(y)).$ 

If $x\leq d(1)<y$, then
$\mathfrak{D}(x)=d(1)$, $ \mathfrak{D}(y)=d(y)$ and
$d(1)\leq d(y)$ by Proposition \mref{pro:3000} \mref{it:3004},
which implies that
$\mathfrak{D}(x\wedge y)=\mathfrak{D}(x)=d(1)=(\mathfrak{D}(x)\wedge y)\vee (x\wedge \mathfrak{D}(y))\vee (\mathfrak{D}(x)\wedge \mathfrak{D}(y))$.

If  $d(1)<x< y$, then
$\mathfrak{D}(x)=d(x), \mathfrak{D}(y)=d(y)$ and
$ d(1) < x\wedge y$. Thus
$\mathfrak{D}(x\wedge y)=d(x\wedge y)=(d(x)\wedge y)\vee (x\wedge d(y))\vee (d(x)\wedge d(y))
=(\mathfrak{D}(x)\wedge y)\vee (x\wedge \mathfrak{D}(y))\vee (\mathfrak{D}(x)\wedge \mathfrak{D}(y))$.

Summarizing the above arguments, we conclude that  $\mathfrak{D}$ is in $\DW_{1}(L)$.
\end{proof}

In Proposition \mref{por:113}, if $L$ is not a chain,
 then  $D$ or $\mathfrak{D}$ may be not in $\DW_{1}(L)$. For example, let
$M_{3}=\{0, b_{1}, b_{2}, b_{3}, 1\}$ be the modular lattice in Example \mref{exa:215}
 and  define $d: M_{3} \to M_{3}$ by
 $d(x):=x \wedge b_{1}$ for all $x\in M_{3}$. Then $d$ is a derivation \mcite{GG}, and so
  $d\in \DW_1(M_3)$ by Proposition \mref{pro:301} \mref{it:3011}, but $D $ and $\mathfrak{D}$ are not in $ \DW_{1}(M_{3})$, since
$D(b_{2}\wedge b_{3})=D(0)=d(0)=0\neq 1=(D(b_{2})\wedge b_{3})\vee (b_{2}\wedge D(b_{3}))\vee (D(b_{2})\wedge D(b_{3}))$ and
$\mathfrak{D}(b_2\wedge b_3)=\mathfrak{D}(0)=d(1)=b_1\neq 0 =(\mathfrak{D}(b_{2})\wedge b_3)\vee (b_2 \wedge \mathfrak{D}(b_3))\vee (\mathfrak{D}(b_2)\wedge \mathfrak{D}(b_{3})).$ 

\section{Difference operators  on finite chains}
\mlabel{sec:chain}

For a positive integer $n$, let
 $L_{n}=\{a_{0}, a_{1}, a_{2}, \cdots,  a_{n-1}\}$ be a $n$-element chain with $$0=a_{0}< a_{1}< a_{2}< \cdots< a_{n-2}< a_{n-1}=1.$$
For  $i,j\in \{0,1,2,\cdots,n-1\}$, denote 
$\Omega_i^j(L_n):=\{d\in \DW_{1}(L_{n}) ~|~d(a_{j})=a_{i}\}$ and $ \omega_{i}^{j}(L_{n}):=|\Omega_{i}^{j}(L_{n})|$, the cardinality of $\Omega_i^j(L_n)$.

Our goal in this section is to determine the cardinality $|\DW_{1}(L_{n})|$.
In \mcite{Hi}, Higgins investigates combinatorial properties of the semigroup $\mathcal{O}_{n}$ of all isotone mappings on
$L_{n}$ and of its subsemigroup $\mathcal{P}_{n}$
 consisting of all maps that are increasing and isotone. 
 He proves that $|\mathcal{P}_{n}|=C(n)$ for the $n$-th Catalan number
$C(n)=\frac{\tbinom{2n}{n}}{n+1}$~\cite[Theorem 3.1]{Hi}. Then together with Corollary~\mref{cor:302}, we immediately have 
 
 \begin{lemma}
 $\omega_{n-1}^{n-1}(L_{n})=C(n)$.
 \mlabel{lem:400}
 \end{lemma}

We also have 
\begin{lemma}
Let $n\geq 3$ and $d$ an operator on $L_{n}$ such that $a_{n-2}\leq d(a_{n-2})$. If $d\in \DW_{1}(L_{n})$,  then $x\leq d(x)$ for all $x\in L_{n}\backslash \{1\}$.
Furthermore, $d\in \DW_{1}(L_{n})$ if and only if $d^1\in \DW_{1}(L_{n})$ for the operator $d^1$ defined in Eq.~\meqref{eq:du}.
\mlabel{lem:401}
\end{lemma}
\begin{proof}
Let $d\in \DW_{1}(L_{n})$  such that $a_{n-2}\leq d(a_{n-2})$.
For each $x\in L_{n}\backslash \{1\}$, we have $x\leq a_{n-2}$, and so
$x=x\wedge d(a_{n-2})\leq d(x\wedge a_{n-2})=d(x)$ by Eq.\meqref{eq:301}.
Then the last statement follows immediately from Proposition \mref{por:111} \mref{it:113}.
\end{proof}

\begin{lemma}
For each $0\leq j\leq n-1$, denote $$A_{j}:=\{d\in\Omega_{j}^{n-1}(L_{n})~|~a_{n-2}\leq d(a_{n-2})\}.$$ 
Then
$|A_{j}|=C(n)$. In particular,
	$\omega_{n-2}^{n-1}(L_{n})=C(n)$.
	\mlabel{lem:403}
\end{lemma}
\begin{proof} 
Recall that $\Omega_{n-1}^{n-1}(L_{n})=\{d\in \DW_{1}(L_{n}) ~|~d(1)=1\}$.
For $d\in A_j$, define $f(d)=d^1$ where $d^1$ is defined in
 Eq.\meqref{eq:du}.
We show this defines a map  $f: A_{j}\to \Omega_{n-1}^{n-1}(L_{n})$. In fact, for
	each  $d\in A_{j}$, i.e, $d\in \DW_{1}(L_{n})$ such that $d(1)=a_{j}$ and $a_{n-2}\leq d(a_{n-2})$, we have
	$d^1\in \Omega_{n-1}^{n-1}(L_{n})$ by Lemma \mref{lem:401}.
	
Define  $g: \Omega_{n-1}^{n-1}(L_{n})\to A_{j}$ by
	$g(d)=d^{a_{j}}$ for all $d\in \Omega_{n-1}^{n-1}(L_{n})$, where $d^{a_{j}}$ is defined in 
	Eq.\meqref{eq:du}. Then $g$ is a map. In fact, let $d\in \Omega_{n-1}^{n-1}(L_{n})$. Then  $d(1)=1$, and so 	$g(d)=d^{a_{j}}\in \DW_{1}(L_n)$ by Proposition \mref{por:111} \mref{it:111}; and we get
	$x\leq d(x)$ for all $x\in L_n 
	 $ by Proposition \mref{pro:3000}
	 \mref{it:3003}. Thus $g(d)=d^{a_{j}}\in A_{j}$, and hence $g$ is a map.
	
Also,	it is easy to check that $fg=\mrep_{\Omega_{n-1}^{n-1}(L_{n})}$ and $gf=\mrep_{A_{j}}$.
	Thus $f$ is a bijection, and consequently
	$|A_{j}|=|\Omega_{n-1}^{n-1}(L_{n})|=\omega_{n-1}^{n-1}(L_{n})	=C(n)$ 	
	by Lemma \mref{lem:400}.
	
	In particular, 
	$A_{n-2}=\Omega_{n-2}^{n-1}(L_{n})$
	by Proposition \mref{pro:3000} \mref{it:3001}, so 	$\omega_{n-2}^{n-1}(L_{n})=C(n)$.
\end{proof}

\begin{lemma}
Let $n\geq 3$ and $d$ an operator on $L_{n}$ such that  $d(1)\neq 1$ and
$d(1)\leq d(a_{n-2})\leq a_{n-3}$. Then $d\in \DW_{1}(L_{n})$ if and only if $d|_{ L_{n-1}}\in \DW_{1}(L_{n-1})$, where
$L_{n-1}=L_{n}\backslash \{1\}=\{0, a_{1}, a_{2}, \cdots, a_{n-2}\}$ and $d|_{ L_{n-1}}$ is the restriction of $d$ to the chain $L_{n-1}$. 
	\mlabel{lem:404}
\end{lemma}
\begin{proof}
	Let $d$ be as given in the lemma. 
	Assume that $d\in \DW_{1}(L_{n})$.
	If there exists $b\in L_{n-1}$ such that $d(b)=1$, then 
	$a_{n-2}\vee d(a_{n-2})=1$ by Proposition \mref{pro:3000} \mref{it:3006}, and so
	$d(a_{n-2})=1$, a contradiction.
	Thus $d|_{ L_{n-1}}$ is an operator on $L_{n-1}$, and hence
	 $d|_{ L_{n-1}}\in \DW_{1}(L_{n-1})$.
	
Conversely, if  $d|_{ L_{n-1}}\in \DW_{1}(L_{n-1})$, then $d(x)\leq a_{n-2}$ for all  $x\in L_{n-1}$, since $d(x)\in L_{n-1}$.
 To prove that $d\in \DW_{1}(L_{n})$, it suffices to show that $x\wedge d(1)\leq d(x)$ for all $x\in L_{n-1}$.
 
In fact, let $x\in L_{n-1}$. 
If
$x\leq d(a_{n-2})$, then, since $a_{n-2}$ is the top element of $L_{n-1}$, we have
$x\leq d(x)$  by Proposition \mref{pro:3000} \mref{it:3001}, and so $x\wedge d(1)\leq x\leq d(x)$.
If $x>d(a_{n-2})$, then	$d(x)\geq d(a_{n-2})\geq d(1)$ by Proposition \mref{pro:3000} \mref{it:3004}, and so $x\wedge d(1)\leq d(1)\leq d(x)$.	Thus  $x\wedge d(1)\leq d(x)$ for all $x\in L_{n-1}$, and hence  $d\in \DW_{1}(L_{n})$.
\end{proof}

\begin{lemma}
Let $n\geq 3$ and  $0\leq j\leq n-3$. Then $\omega_{j}^{n-1}(L_{n})=C(n)+\sum_{k=j}^{n-3}\omega_{k}^{n-2}(L_{n-1})$. In particular, 	$\omega_{n-3}^{n-1}(L_{n})=C(n)+C(n-1)$.
	\mlabel{lem:405}
\end{lemma}
\begin{proof}
	Let $0\leq j\leq n-3$ and	$d\in\Omega_{j}^{n-1}(L_{n})$, i.e,
	$d\in \DW_{1}(L_{n})$ with $d(1)=a_{j}$. Then $d(1)\leq a_{n-2}$, and so
	$d(1)\leq d(a_{n-2})$ by Proposition \mref{pro:3000} \mref{it:3004}.
	
If $d(a_{n-2})\geq a_{n-2}$, then $d$ has $C(n)$ choices by Lemma \mref{lem:403}.
If $d(a_{n-2})\leq a_{n-3}$, then $d$ has $\omega_{k}^{n-2}(L_{n-1})$ choices by Lemma \mref{lem:404}, where $d(1)=a_{j}\leq d(a_{n-2})=a_{k}\leq a_{n-3}$.
Thus we conclude that $\omega_{j}^{n-1}(L_{n})=C(n)+\sum_{k=j}^{n-3}\omega_{k}^{n-2}(L_{n-1})$.

In particular, we have
$\omega_{n-3}^{n-1}(L_{n})=C(n)+\omega_{n-3}^{n-2}(L_{n-1})=C(n)+C(n-1)$ by Lemma \mref{lem:403}.
\end{proof}

 \begin{proposition}
	\mlabel{pro:402}
Let $n\geq 3$ and  $0\leq j\leq n-3$. Then
	\begin{equation}
	\mlabel{eq:400}	
\omega_{j}^{n-1}(L_{n})=C(n)+\sum_{k=1}^{n-2-j}\tbinom{n-2-j}{k}C(n-k). 
	\end{equation}	
\end{proposition}
\begin{proof}
	Let  $0\leq j\leq n-3$ and $t=n-2-j$. We prove it on $t$ by induction. 
	
When $t=1$, i.e, $j=n-3$, we know from Lemma \mref{lem:405} that $\omega_{n-3}^{n-1}(L_{n})=C(n)+C(n-1)$. Thus Eq.\meqref{eq:400} holds for $t=1$.		
 Suppose that Eq.\meqref{eq:400} holds for $t=q$, i.e,
 	\begin{equation}
\omega_{n-2-q}^{n-1}(L_{n})=C(n)+\sum_{k=1}^{q}\tbinom{q}{k}C(n-k).
 \mlabel{eq:4000}	
 \end{equation}
 Then we also have
	\begin{equation}
\omega_{n-3-q}^{n-2}(L_{n-1})=C(n-1)+\sum_{k=1}^{q}\tbinom{q}{k}C(n-1-k).
\mlabel{eq:4001}	
\end{equation} 

We will show that  Eq.\meqref{eq:400} holds for $t=q+1$, i.e, 
  for $j=n-3-q$.
 In fact, we have  by Lemma \mref{lem:405} that
 $\omega_{n-2-q}^{n-1}(L_{n})=C(n)+\sum_{k=n-2-q}^{n-3}\omega_{k}^{n-2}(L_{n-1})$,
 which, by using Lemma \mref{lem:405} again and the inductive assumption, implies that
 \begin{eqnarray*}
 \omega_{n-3-q}^{n-1}(L_{n})
 &=&C(n)+\sum_{k=n-3-q}^{n-3}\omega_{k}^{n-2}(L_{n-1})\\
 &=& (C(n)+\sum_{k=n-2-q}^{n-3}\omega_{k}^{n-2}(L_{n-1}))+\omega_{n-3-q}^{n-2}(L_{n-1})\\
 &=&\omega_{n-2-q}^{n-1}(L_{n})+\omega_{n-3-q}^{n-2}(L_{n-1})\\
 &=& (C(n)+\sum_{k=1}^{q}\tbinom{q}{k}C(n-k))+(C(n-1)+\sum_{k=1}^{q}\tbinom{q}{k}C(n-1-k)) \quad (\textrm{by Eqs. \meqref{eq:4000}  and \meqref{eq:4001}})\\
 &=&C(n)+\sum_{k=1}^{q}(\tbinom{q}{k}+\tbinom{q}{k-1})C(n-k))+C(n-1-q)\\
 &=&C(n)+\sum_{k=1}^{q+1}\tbinom{q+1}{k}C(n-k)\quad\quad
 (\textrm{by }\tbinom{q}{k}+\tbinom{q}{k-1})
=\tbinom{q+1}{k})\\
&=&C(n)+\sum_{k=1}^{n-2-(n-3-q)}\tbinom{n-2-(n-3-q)}{k}C(n-k).
 \end{eqnarray*}	
Thus Eq.\meqref{eq:400} holds for $t=q+1$, and so  Eq.\meqref{eq:400} holds for all $0\leq j\leq n-3$ by induction. 
\end{proof}

\begin{theorem}
Let $n$ be a positive integer. Then
$$
|\DW_{1}(L_{n})|=
\begin{cases}
nC(n),  & \textrm{if}~ n\in \{1, 2\}; \\
nC(n)+\sum_{k=1}^{n-2}
\sum_{p=k}^{n-2}\tbinom{p}{k}C(n-k),  & \textrm{if}~ n \geq 3.
\end{cases}
$$		
	\mlabel{the:400}
\end{theorem}
\begin{proof}
Since $|\DW_{1}(L_{n})|=\sum_{j=0}^{n-1}
\omega_{j}^{n-1}(L_{n})$,
	it follows immediately from Lemma \mref{lem:400}, Lemma \mref{lem:403} and Proposition \mref{pro:402}.
\end{proof}

We list the first few terms of $|\DW_{1}(L_{n})|$ for $1\leq n\leq 10$.
Recall that the corresponding Catalan numbers $C(n)$ are $1, 2, 5, 14, 42, 132, 429$, $1430$, $4862, 16796$.
Then Theorem~\mref{the:400} gives 
\smallskip 

\begin{center}
\begin{tabular}{|c|c|c|c|c|c|c|c|c|c|c|}
	\hline $n$ & 1 &2 & 3 & 4 &5 & 6 & 7&8 &9 &10\\
	\hline $|\DW_{1}(L_{n})|$ & 1 &4 & 17 & 73 &316 & 1379 & 6065&26870&119848&537877\\
	\hline
\end{tabular}
\end{center}
\smallskip 
This integer sequence is not in the database OEIS~\mcite{OEIS}. 

From the above table, we know that
 $\DW_{1}(L)=\{d ~|~d$ is an operator on $L\}$ if $L$ is a lattice with $|L|\leq 2$.
As shown below, these are the only cases when this equality hold. 
\begin{proposition}
	Let $L$ be a lattice. Then $\DW_{1}(L)=\{d ~|~d$ is an operator on $L\}$ if and only if $|L|\leq 2$.
	\mlabel{pro:410}
\end{proposition}
\begin{proof}
It suffices to show that $\DW_{1}(L)\neq\{d ~|~d$ is an operator on $L\}$ when
$|L|\geq 3$. In fact, let $|L|\geq 3$.
If $L$ is a chain, then there 
exist $a, b, c\in L$ such that $a<b<c$. Define an operator $\mathbb{D}: L\to L$ by
$$
\mathbb{D}(x):=
\begin{cases}
c,  & \textrm{if}~ x= a; \\
a,  & \textrm{if}~ x= c; \\
x,  & \textrm{otherwise}.
\end{cases}
$$	
Then $\mathbb{D}\not\in \DW_{1}(L) $, since $\mathbb{D}(a\wedge b)=\mathbb{D}(a)=c\neq b=(\mathbb{D}(a)\wedge b)\vee (a\wedge \mathbb{D}(b))\vee (\mathbb{D}(a)\wedge \mathbb{D}(b))$.

If $L$ is not a chain, then there exist $a, b\in L$ such that $a$ and $b$ are incomparable. Define an operator $d: L\to L$ by
$$
d(x):=
\begin{cases}
a,  & \textrm{if}~ x\leq a; \\
b,  & \textrm{otherwise}.
\end{cases}
$$	
Then $d\not\in \DW_{1}(L) $ since $d(a\wedge b)=a\neq a\wedge b=(d(a)\wedge b)\vee (a\wedge d(b))\vee (d(a)\wedge d(b))$.
\end{proof}

\section{Difference operators on quasi-antichains}
\mlabel{sec:qua}

This section classifies and enumerates \diffcs on quasi-antichains $M_{n-2}, n\geq 4$, with $n-2$ atoms. The classification for $M_2$ will be given with the enumerations in Theorem~\mref{the:414}. 
For $M_{n-2}, n\geq 4$, we will distinguish the three cases when the \diffc $d$ satisfies $d(0)=0$, $d(0)=1$ or $d(0)\neq 0,1$. 

The classification of the \diffcs in the case when $d(0)=0$ is obtained in Proposition~\mref{pro:500} in Section~\mref{sss:d=0}. The classifications of the other two cases are more involved. So we first present several families of \diffcs in these two cases in Sections~\mref{sss:d=1} and \mref{sss:dother} respectively. We then achieve the complete classifications in these two cases in Sections~\mref{sss:cd=1} and \mref{sss:cdother} respectively. The enumeration of the \diffcs on $M_{n-2}$ is given in Section~\mref{sss:enum}. 

\subsection{Families of \diffcs on a  quasi-antichain}
\mlabel{ss:special}
We will carry out the study in three parts. 
\subsubsection{Notations and the classification of \diffcs when $d(0)=0$}
\mlabel{sss:d=0}
A {\bf quasi-antichain}~\cite{Br} is a lattice consisting of a top element $1$, a bottom element $0$, and the atoms of the lattice. A $n$-element quasi-antichain is a chain if $n$ is $2$ or $3$. 
When $n\geq 4$, the
$n$-element quasi-antichain  $M_{n-2}=\{0, b_{1}, b_{2}, \cdots, b_{n-2}, 1\}$ with $n-2$ atoms  has Hasse diagram

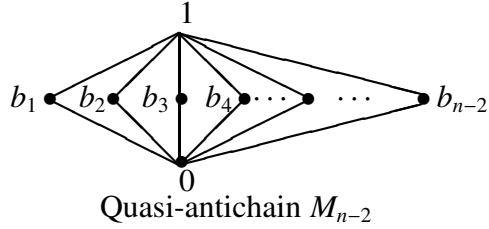
\begin{center}
	\setlength{\unitlength}{0.7cm}
	\begin{picture}(4,5)
	\thicklines
	\put(2.0,1.0){\line(-2,1){2.5}}
	\put(2.0,1.0){\line(-1,1){1.2}}
	\put(2.0,1.0){\line(0,1){1.2}}
	\put(2.0,1.0){\line(1,1){1.2}}
	\put(2.0,1.0){\line(2,1){2.5}}
	\put(2.0,1.0){\line(4,1){4.6}}
	\put(2.0,3.5){\line(-2,-1){2.5}}
	\put(2.0,3.5){\line(-1,-1){1.2}}
	\put(2.0,3.5){\line(0,-1){1.2}}
	\put(2.0,3.5){\line(1,-1){1.2}}
	\put(2.0,3.5){\line(2,-1){2.5}}
	\put(2.0,3.5){\line(4,-1){4.6}}
	\put(2.0,0.5){$0$}
	\put(-1.2,2.1){$b_{1}$}
	\put(0.1,2.1){$b_{2}$}
	\put(1.3,2.1){$b_{3}$}
	\put(2.5,2.1){$b_{4}$}
	\put(3.4,2.1){$\cdots$}
	\put(5.0,2.1){$\cdots$}
	\put(6.9,2.1){$b_{n-2}$}
	\put(2.0,3.7){$1$}
	\put(1.9,0.9){$\bullet$}
	\put(-0.6,2.1){$\bullet$}
	\put(0.6,2.1){$\bullet$}
	\put(1.9,2.1){$\bullet$}
	\put(3.1,2.1){$\bullet$}
	\put(4.3,2.1){$\bullet$}
	\put(6.5,2.1){$\bullet$}
	\put(0.5,0.0){Quasi-antichain $M_{n-2}$}
	\end{picture}
\end{center}
It is clear that $M_{2}$ is the $4$-element Boolean algebra.
In this section, unless otherwise statement,  $M$ will denote a quasi-antichain with $|M|\geq 4$. 

Let us first consider \diffcs $d$ on $M$ such that $d(0)=0$.
\begin{proposition}
The set $\DO(M)$ coincides with the subset of $\DW_1(M)$ with $d(0)=0$, that is, $\DO(M)=\{d\in \DW_{1}(M)~|~d(0)=0\}$.	
\mlabel{pro:500}
\end{proposition}
\begin{proof}
If  $d\in \DO(M)$, then  $d\in \DW_{1}(M)$ by Proposition \mref{pro:301} \mref{it:3011}, and
$d(0)=0$ by Lemma \mref{pro:201} \mref{it:2011}.

Conversely, if $d\in \DW_{1}(M)$ such that $d(0)=0$, then $d(x)\leq x$ for all $x\in M$. In fact, we have $d(0)\leq 0$ and $d(1)\leq 1$.
If there exsits $a\in  M\backslash \{0, 1\}$ 
such that $d(a)\not\leq a$, then
$d(a)=1$ or
$d(a)=b$ for some $b\in  M\backslash \{0, a, 1\}$, which implies  by Eq.\meqref{eq:301} that
$b=d(a)\wedge b\leq d(a\wedge b)=d(0)=0$, a contradiction. Thus
$d(x)\leq x$ for all $x\in M$, and hence $d\in \DO(M)$ by Proposition \mref{pro:301} \mref{it:3011}.	
\end{proof}

\subsubsection{Families of \diffcs when $d(0)=1$}
\mlabel{sss:d=1}
We next consider the case of $d\in \DW_1(M)$ when $d(0)=1$. 
If $d(1)=1$ also, then $d$ is the constant operator $\textbf{C}_{(1)}$ by Proposition \mref{pro:001}.
The case when $d(1)\neq 1$ are more complicated. We will give several classes of \diffcs in
Lemmas \mref{lem:501} and \mref{lem:506} before arriving at a classification in Theorem~\mref{pro:502}.

\begin{lemma}\mlabel{lem:501}
\begin{enumerate}
	\item 
	For $b\in M\backslash\{0, 1\}$,
	$u\in M\backslash\{0, b, 1\}$ and $v\in M\backslash\{0, u, 1\}$,
	define an operator $\eta_{u\to v}$  on $M$ by
	$$
	\eta_{u\to v}(x):=
	\begin{cases}
	b,  & \textrm{if}~ x= 1; \\
	v,  & \textrm{if}~ x=u; \\
	1,  & \textrm{otherwise}.
	\end{cases}
	$$			
	Then
	$\eta_{u \to v}$ is in $\DW_{1}(M)$. In particular, if $v=b$, then $\beta_{b}$ is in $\DW_{1}(M)$, where $\beta_{b}$ is defined on $M$ by
	$$
	\beta_b(x):=
	\begin{cases}
	b,  & \textrm{if}~ x\in \{u, 1\};\\ 
	1,  & \textrm{otherwise}.
	\end{cases}
	$$ \mlabel{it:5011}
	
\item 		Let $M$ be a quasi-antichain with
$|M|\geq 5$, and let
$b, u$ and $v$ be mutually distinct elements in $ M\backslash\{0, 1\}$.
Define an operator $\gamma^{u \leftrightarrow v}$  on $M$ by
$$
\gamma^{u \leftrightarrow v}(x):=
\begin{cases}
b,  & \textrm{if}~ x= 1; \\
v,  & \textrm{if}~ x=u; \\
u,  & \textrm{if}~ x=v; \\
1,  & \textrm{otherwise}.
\end{cases}
$$			
Then
$\gamma^{u \leftrightarrow v}$ is in $\DW_{1}(M)$. \mlabel{it:5012}
\end{enumerate}	
			
\end{lemma}
\begin{proof}
	Let $b, u, v$ and $\xi$ in $\{\eta_{u \to v}, \gamma^{u \leftrightarrow v}\}$ be as given in the lemma.
	To prove $\xi\in \DW_{1}(M)$, let $x, y\in M$.  
Then $x\vee \xi(x)=y\vee \xi(y)=1$ by the definition of $\xi$. 
By Lemma~\mref{lem:x=y}, we can assume that $x\neq y$.

If $x\wedge y\neq 0$, then $x=1$ or $y=1$. Without loss of generality, let $x=1$ and $y\in  M\backslash \{0, 1\}$. Then $\xi(x)=b$ and $\xi(b)=1$
by the definition of $\xi$. Since $\xi(b)=1=(b\wedge b)\vee \xi(b)$	and $b\wedge c=0$ when $c\in M\backslash \{1, b\}$, we get
$\xi(y)=(b\wedge y)\vee \xi(y)$, and so
$$\xi(x\wedge y)=\xi(y)=(b\wedge y)\vee \xi(y)=(\xi(x)\wedge y)\vee (x\wedge \xi(y))\vee (\xi(x)\wedge \xi(y)).$$

If $x\wedge y=0$, we give separate arguments for the two parts of the lemma. 

For \mref{it:5011}, since $x\wedge y=0$, by the definition of $\eta_{u\to v}$, we have  $\eta_{u\to v}(x)=1 $ or $\eta_{u\to v} (y)=1$. It follows that
	$$\eta_{u\to v}(x\wedge y)=\eta_{u\to v}(0)=1=(\eta_{u\to v}(x)\wedge y)\vee (x\wedge \eta_{u\to v}(y))\vee (\eta_{u\to v}(x)\wedge \eta_{u\to v}(y)),$$
	since $x\vee \eta_{u\to v}(x)=y\vee \eta_{u\to v}(y)=1$.
	
This completes the proof that $\eta_{u\to v}$ is in $ \DW_{1}(M)$.

For \mref{it:5012}, we distinguish two cases. If $x\wedge y=0$, and $\gamma^{u \leftrightarrow v}(x)=1 $ or $\gamma^{u \leftrightarrow v} (y)=1$, then
$$\gamma^{u \leftrightarrow v}(x\wedge y)=\gamma^{u \leftrightarrow v}(0)=1=(\gamma^{u \leftrightarrow v}(x)\wedge y)\vee (x\wedge \gamma^{u \leftrightarrow v}(y))\vee (\gamma^{u \leftrightarrow v}(x)\wedge \gamma^{u \leftrightarrow v}(y)),$$
since $x\vee \gamma^{u \leftrightarrow v}(x)=y\vee \gamma^{u \leftrightarrow v}(y)=1$.

If $x\wedge y=0$,  $\gamma^{u \leftrightarrow v}(x)\neq 1 $ and $\gamma^{u \leftrightarrow v} (y)\neq 1$, then $\{x, y\}=\{u, v\}$, and so
$\gamma^{u \leftrightarrow v}(x\wedge y)=\gamma^{u \leftrightarrow v}(0)=1=u\vee v=(\gamma^{u \leftrightarrow v}(u)\wedge v)\vee (u\wedge \gamma^{u \leftrightarrow v}(v))\vee (\gamma^{u \leftrightarrow v}(u)\wedge \gamma^{u \leftrightarrow v}(v))=(\gamma^{u \leftrightarrow v}(x)\wedge y)\vee (x\wedge \gamma^{u \leftrightarrow v}(y))\vee (\gamma^{u \leftrightarrow v}(x)\wedge \gamma^{u \leftrightarrow v}(y))$.

In summary, we conclude that $\gamma^{u \leftrightarrow v}$ is in $ \DW_{1}(M)$.
\end{proof}

\begin{lemma}
	Let $u, v\in M\backslash\{0, 1\}$ with $u\neq v$.
	Define  operators $\theta_{u\to v}$ and $\alpha^{u \leftrightarrow v}$  on $M$, respectively, by
	$$
	\theta_{u\to v}(x):=
	\begin{cases}
	0,  & \textrm{if}~ x= 1; \\
	v,  & \textrm{if}~ x=u; \\
	1,  & \textrm{otherwise}.
	\end{cases}	\quad\quad and \quad\quad
	\alpha^{u \leftrightarrow v}(x):=
	\begin{cases}
	0,  & \textrm{if}~ x= 1; \\
	v,  & \textrm{if}~ x=u; \\
	u,  & \textrm{if}~ x=v; \\
	1,  & \textrm{otherwise}.
	\end{cases}
	$$			
	Then 
	$\theta_{u\to v}$ and $\gamma^{u \leftrightarrow v}$ are in $ \DW_{1}(M)$.
	\mlabel{lem:506}
\end{lemma}
\begin{proof}
	Let $u, v, \theta_{u\to v}$ and $\alpha^{u \leftrightarrow v}$ be as given in the lemma.
	Fix	$b\in M\backslash \{0, u, 1\}$.
	Then $\theta_{u\to v}=(\eta_{u\to v})^{0}$, and so 	$\theta_{u\to v}\in \DW_{1}(M)$ by Proposition \mref{por:111} and Lemma \mref{lem:501} \ref{it:5011}.
	
	It is routine to check that $\alpha^{u \leftrightarrow v}\in \DW_{1}(M)$ if $|M|=4$.	
	When $|M|\geq 5$, fix $b\in M\backslash \{0, 1, u, v\}$. Then
	$\alpha^{u \leftrightarrow v}=(\gamma^{u \leftrightarrow v})^{0}$, and so 
	$\alpha^{u \leftrightarrow v}\in \DW_{1}(M)$ by Proposition \mref{por:111} and Lemma \mref{lem:501} \ref{it:5012}.
\end{proof}

\subsubsection{Families of \diffcs when $d(0)\neq 0,1$}
\mlabel{sss:dother}
Finally, we present a class of \diffcs $d$ on $M$ such that $d(0)\in M\backslash\{0, 1\}$.
\begin{lemma}
	Let $a\in M\backslash\{0, 1\}$.
	Define an operator $\lambda_{a}$  on $M$ by
	$$
	\lambda_{a}(x):=
	\begin{cases}
	0,  & \textrm{if}~ x\in \{a, 1\}; \\
	a,  & \textrm{otherwise}.
	\end{cases}
	$$			
	Then
	$\lambda_{a}$ is in $ \DW_{1}(M)$. 			
	\mlabel{lem:511}
\end{lemma}
\begin{proof}
	Let $a$ and $\lambda_{a}$ be as given in the lemma.
	To prove that $\lambda_{a}\in \DW_{1}(M)$, let $x, y\in M$ which we can assume to be distinct because of Lemma~\mref{lem:x=y}.  
	
	If $x, y\in M\backslash \{a, 1\}$, then $x\wedge y=0$, and $\lambda_{a}(x)=\lambda_{a}(y)=a $
	by the definition of $\lambda_{a}$. So
$$\lambda_{a}(x\wedge y)=\lambda_{a}(0)=a=(\lambda_{a}(x)\wedge y)\vee (x\wedge \lambda_{a}(y))\vee (\lambda_{a}(x)\wedge \lambda_{a}(y)).$$
	
If $x=1$ or $y=1$ (but not both), say $x=1$, then $x\wedge y=y$, and $\lambda_{a}(x)=0$
	by the definition of $\lambda_{a}$. It follows that
	$\lambda_{a}(x\wedge y)=\lambda_{a}(y)=(\lambda_{a}(x)\wedge y)\vee (x\wedge \lambda_{a}(y))\vee (\lambda_{a}(x)\wedge \lambda_{a}(y)).$
	
	If $x=a$ or $ y=a$, say $x=a$ and $ y\in M\backslash \{a, 1\}$, then $x\wedge y=0$, and $\lambda_{a}(x)=0$, $\lambda_{a}(y)=a $
	by the definition of $\lambda_{a}$. Thus
	$\lambda_{a}(x\wedge y)=\lambda_{a}(0)=a=(\lambda_{a}(x)\wedge y)\vee (x\wedge \lambda_{a}(y))\vee (\lambda_{a}(x)\wedge \lambda_{a}(y)).$
	
	Therefore we conclude that $\lambda_{a}$ is in $ \DW_{1}(M)$.
\end{proof}

 For  a finite set $X=\{x_{1}, x_{2}, \cdots, x_{n}\}$, and a map $\varphi: X\rightarrow X$. We write
$$ \left( \begin{matrix}
x_{1}          & x_{2}            & \cdots & x_{n} \\
\varphi(x_{1}) &   \varphi(x_{2}) & \cdots &   \varphi(x_{n})
\end{matrix}
\right)$$
for $\varphi$. For convenience, let
$M_{2}=\{0, a, b, 1\}$ denote the $4$-element quasi-antichain.

\begin{lemma}
	Define operators  $\Phi_{a}$, $\Phi_{b}$ and $\Theta$ respectively  on $M_{2}$ by 
	$$ \Phi_{a}=\left( \begin{matrix}
	0   & a    & b & 1 \\
	a   &   0 & 1 &  b
	\end{matrix}
	\right), \quad
\Phi_{b}=\left( \begin{matrix}
0   & a    & b & 1 \\
b   &   1 & 0 &  a
\end{matrix}
\right), \quad	
	\Theta=\left( \begin{matrix}
	0   & a    & b & 1 \\
	a    &   0 & a &  b
	\end{matrix}
	\right). $$
	Then
	$\Phi_{a}$ and $\Phi_{b}$ are in $  \DW_{1}(M_{2})$, but $\Theta\not\in \DW_{1}(M_{2})$. 			
	\mlabel{lem:513}
\end{lemma}
\begin{proof}
To prove that $\Phi=\Phi_{a}$ is in $\DW_{1}(M_{2})$, we let $x, y\in M_{2}$ and assume that they are distinct thanks to Lemma~\mref{lem:x=y}. 

If $x=1$ or $y=1$ (but not both), say $x=1$ and $y\neq 1$, then  $\Phi(x)\wedge y\leq \Phi(y)$ by the definition of $\Phi$, and so 	$\Phi(x\wedge y)=\Phi(y)=(\Phi(x)\wedge y)\vee \Phi(y)=(\Phi(x)\wedge y)\vee (x\wedge \Phi(y))\vee (\Phi(x)\wedge \Phi(y))$.
	
If $x=0$ or $y=0$, say $x=0$ and $y\in \{a, b\}$, then it is easy to check that
	$(a\wedge y)\vee(a\wedge \Phi(y))=a$, and so
	$\Phi(x\wedge y)=\Phi(0)=a=(a\wedge y)\vee(a\wedge \Phi(y))=(\Phi(x)\wedge y)\vee (x\wedge \Phi(y))\vee (\Phi(x)\wedge \Phi(y))$.
	
	If  $x, y\in \{a, b\}$, say $x=a$ and $y=b$, then $\Phi(x)=0$ and $\Phi(y)=1$, which implies that
	$\Phi(x\wedge y)=\Phi(0)=a=a\wedge \Phi(b)=(\Phi(x)\wedge y)\vee (x\wedge \Phi(y))\vee (\Phi(x)\wedge \Phi(y))$.	
	
Therefore we conclude that $\Phi=\Phi_{a}$ is in $ \DW_{1}(M_{2})$. Interchanging $a$ with $b$, we can prove that $\Phi_{b}$ is in $ \DW_{1}(M_{2})$.
	
Finally, suppose on the contrary that	 $\Theta\in \DW_{1}(M_{2})$. Then we have
	by Eq.\meqref{eq:301} that
	$b=\Theta(1)\wedge b\leq \Theta(1\wedge b)=\Theta(b)=a,$
	a contradiction.	Thus $\Theta\not\in \DW_{1}(M_{2})$. 
\end{proof}

\subsection{The classification and enumeration of \diffcs}
\mlabel{ss:main}
The classification of \diffcs when $d(0)=0$ has been achieved in Proposition~\mref{pro:500}. So we only need to consider the cases when $d(0)=1$ and $d(0)\neq 0,1$ which we will treat in Sections~\mref{sss:cd=1} and \mref{sss:cdother}. The enumeration of the \diffcs will be given in Section~\mref{sss:enum}. 

\subsubsection{Classification of \diffcs when $d(0)=1$}
\mlabel{sss:cd=1}
To classify and enumerate \diffcs on quasi-antichains, the following lemmas are needed.

\begin{lemma}
	Let $u, v\in M\backslash\{0, 1\}$ with $u\neq v$.
	If $d$ is an operator on $M$ such that
	$d(0)=1$ and $d(u)=d(v)\neq 1$, then $d\not\in \DW_{1}(M)$.
	\mlabel{lem:504}
\end{lemma}
\begin{proof}
	Let $ u, v$ and $d$ be as given in the lemma.
	Then 
	$d(u\wedge v)=d(0)=1\neq d(u)= (d(u)\wedge v)\vee (u\wedge d(v))\vee (d(u)\wedge d(v)),$ 
	and so 	$d\not\in \DW_{1}(M)$.
\end{proof}

\begin{lemma}
	Let $|M|\geq 5$, and
	$u, v$, $w$ be mutually distinct elements in $ M\backslash\{0, 1\}$.
	If $d$ is an operator on $M$ such that
	$d(0)=1$ and $d(u), d(v), d(w)\in M\backslash \{0, 1\}$, then $d\not\in \DW_{1}(M)$.
	\mlabel{lem:505}
\end{lemma}
\begin{proof}
	Let $ u, v, w$ and $d$ be as given in the lemma.
	
	If $d(u)=d(v)$ or $d(v)=d(w)$ or
	$d(u)=d(w)$, then $d\not\in \DW_{1}(M)$ by Lemma \mref{lem:504}.
	
	If  $d\in \DW_{1}(M)$ and $d(u), d(v), d(w)$ are mutually distinct elements in $M\backslash \{0, 1\}$,	
	then  $d(u)\wedge d(v)=0$, and so
	$1=d(0)=d(u\wedge v)=(d(u)\wedge v)\vee (u\wedge d(v))\vee (d(u)\wedge d(v))=(d(u)\wedge v)\vee (u\wedge d(v)).$
	It follows  that $d(u)=v$	and $d(v)=u$. 
	Similarly, we can obtain that $d(u)=w$ and $d(w)=u$. Thus $d(u)=v=w$, a contradiction. Consequently, $d$ is not in $\DW_{1}(M)$.
\end{proof}	
	
	\begin{lemma}
Let $d\in \DW_{1}(M)$ such that
$d(0)=1$. Then there exist at most two elements $u, v\in M\backslash \{0, 1\}$
such that $d(u), d(v)\in M\backslash \{0,  1\}$, and $d(y)=1$ for all $y\in M\backslash \{u, v, 1\}$.
		\mlabel{lem:50}
\end{lemma}
\begin{proof}
Let $d\in \DW_{1}(M)$ such that
$d(0)=1$. 
By Lemma \mref{lem:505},
there exist at most two elements $u, v\in M\backslash \{0, 1\}$
such that $d(u), d(v)\in M\backslash \{0, 1\}$.

If  $ M\backslash \{0, u, v, 1\}=\emptyset$, then we have done, since $d(0)=1$.

If  $ M\backslash \{0, u, v, 1\}\neq\emptyset$, then
for each $y\in M\backslash \{0, u, v, 1\}$,  we have $y\vee d(y)=1$ by Proposition 
\mref{pro:3000} \mref{it:3005}, so 
$d(y)\in M\backslash \{0, y\}$. It follows that $d(y)=1$, since $d(y)\not\in M\backslash \{0, 1\}$.
\end{proof}

Now we classify all  \diffcs  $d$ on  $M$ such that $d(0)=1$.
\begin{theorem}
All \diffcs $d$  on $M$ such that $d(0)=1$ are given as follows.
\begin{enumerate}
	\item When $d(1)=1$, then $d=\textbf{C}_{(1)}$. \mlabel{it:8111}
	\item When $d(1)=0$, then $d$ is  of the following forms:
	\begin{enumerate}
	\item  $(\textbf{C}_{(1)})^{0}$ (as defined in Eq.\meqref{eq:du});
	\item $\theta_{u \to v}$ for some $u, v\in M\backslash \{0, 1\}$ with $u\neq v$  (as defined in Lemma \mref{lem:506});
	\item $\alpha^{u \leftrightarrow v}$ for some $u, v\in M\backslash \{0, 1\}$ with $u\neq v$ (as defined in Lemma \mref{lem:506}).
\end{enumerate}	
\mlabel{it:8112}	
	\item When $d(1)=b\in M\backslash \{0,1\}$, then $d$ is of the following forms:
	\begin{enumerate}
	\item  $(\textbf{C}_{(1)})^{b}$ (as defined in Eq. \meqref{eq:du});
	\item $\eta_{u \to v}$ for some $u, v\in M\backslash \{0, 1\}$ with $u\neq v$ and $u\neq b$ (as defined in Lemma \mref{lem:501} \ref{it:5011});
	\item $\gamma^{u \leftrightarrow v}$ for some $u, v\in M\backslash \{0, b, 1\}$ with $u\neq v$ (as defined in Lemma \mref{lem:501} \ref{it:5012}).
\end{enumerate}	
\mlabel{it:8113}
\end{enumerate}
\mlabel{pro:502}
\end{theorem}
\begin{proof} 
Let $d\in \DW_{1}(M)$ such that  $d(0)=1$. Then for each $x\in M\backslash \{ 1\}$,  we have $x\vee d(x)=1$ by Proposition 
\mref{pro:3000} \mref{it:3005}, so 
$d(x)\in M\backslash \{0, x\}$.	

	\mnoindent
\mref{it:8111} If $d(1)=1$, then 
$d=\textbf{C}_{(1)}$ by Proposition \mref{pro:001}.
	
	\mnoindent
\mref{it:8112}	When $d(1)=0$, we know by Proposition \mref{por:111} and Lemma \mref{lem:506} that $(\textbf{C}_{(1)})^{0}$, $\theta_{u \to v}$ and $\alpha^{u \leftrightarrow v}$ are in $\DW_{1}(M)$.
Now we  claim that 	 $d$ is  of the three forms. 
In fact,
 by Lemma \mref{lem:50} we only need to consider the following cases.

\smallskip
\noindent
 \textit{(1) $d(x)=1$ for all $x\in M\backslash \{ 1\}$.}	
In this case, it is clear that	  
$d= (\textbf{C}_{(1)})^{0}$, since $d(1)=0$.	
\smallskip

\noindent
\textit{(2) There exists $u\in M\backslash \{0,  1\}$ such that
	$d(u)\neq 1$ and 
	$d(y)=1$ for all $y\in M\backslash \{u, 1\}$. }
In this case, since $d(u)\in M\backslash \{0, u\}$, we have $d(u)\in M\backslash \{0, u, 1\}$. Let $d(u)=v$. We have $v\neq u$ and $d=\theta_{u \to v}$.
\smallskip
	
\noindent
 \textit{(3) There exist $u, v\in M\backslash \{0,  1\}$ with $u\neq v$ such that
	$d(u)\neq 1$, $d(v)\neq 1$ and 
	$d(y)=1$ for all $y\in M\backslash \{u, v, 1\}$.}
In this case, since $d(u)\in M\backslash \{0, u\}$ and $d(v)\in M\backslash \{0, v\}$, we have $d(u)\in M\backslash \{0, u, 1\}$ and  $d(v)\in M\backslash \{0, v, 1\}$.
	We claim that $d(u)=v$ and $d(v)=u$. In fact, we have $d(u)\neq d(v)$ by Lemma \mref{lem:504}, and so $d(u)\wedge d(v)=0$.
	If  $d(u)\neq v$, then $d(u)\wedge v=0$, and so
	 $$d(u\wedge v)=d(0)=1 \neq u \wedge d(v)= (d(u)\wedge v)\vee (u\wedge d(v))\vee (d(u)\wedge d(v)),$$ a contradiction.
	Thus $d(u)=v$. Similarly, we can get that $d(v)=u$. Therefore $d=\alpha^{u \leftrightarrow v}$.

	\mnoindent
\mref{it:8113} When $d(1)=b\in M\backslash \{0,1\}$,
we know by Proposition \mref{por:111} and Lemma \mref{lem:501}  that $(\textbf{C}_{(1)})^{b}$, $\eta_{u \to v}$ and $\gamma^{u \leftrightarrow v}$ are in $\DW_{1}(M)$.

Now we  claim that 	 $d$ is of one of the three forms. 
In fact, since 
$0\leq b= d(1)$, we have  
$1=d(0)\leq d(b)$ by Proposition \mref{pro:3000} \mref{it:3002}, and
so $d(b)=1$.
By Lemma \mref{lem:50} we only need to consider the following cases.
	
\mnoindent
\textit{(1) $d(x)=1$ for all $x\in M\backslash \{1\}$.}
In this case, it is clear that $d= (\textbf{C}_{(1)})^{b}$, since $d(1)=b$.	

\mnoindent
\textit{(2) There exists $u\in M\backslash \{0, b, 1\}$ such that
	$d(u)\neq 1$, and 
	$d(y)=1$ for all $y\in M\backslash \{u, 1\}$.}
In this case, since $d(u)\in M\backslash \{0, 1\}$, we have  $d(u)\in M\backslash \{0, u, 1\}$.
	Let $d(u)=v$. Then $d=\eta_{u \to v}$.

\mnoindent
\textit{(3) There exist $u, v\in M\backslash \{0, b, 1\}$ with $u\neq v$ such that
	$d(u)\neq 1$, $d(v)\neq 1$ and 
	$d(y)=1$ for all $y\in M\backslash \{u, v, 1\}$.}
In this case, since $d(u)\in M\backslash \{0, u\}$ and $d(v)\in M\backslash \{0, v\}$, we have $d(u)\in M\backslash \{0, u, 1\}$ and  $d(v)\in M\backslash \{0, v, 1\}$.	
We claim that $d(u)=v$ and $d(v)=u$. In fact, we have $d(u)\neq d(v)$ by Lemma \mref{lem:504}, and so $d(u)\wedge d(v)=0$.
	If  $d(u)\neq v$, then  $d(u)\wedge v=0$, and so
	$$ (d(u)\wedge v)\vee (u\wedge d(v))\vee (d(u)\wedge d(v))=u \wedge d(v)\neq 1=d(0)=d(u\wedge v),$$ 
	a contradiction.
	Thus $d(u)=v$. Similarly, we can get that $d(v)=u$. Therefore $d=\gamma^{u \leftrightarrow v}$.
\end{proof}

\subsubsection{Classification of \diffcs when $d(0)\neq 0,1$}
\mlabel{sss:cdother}

Next we classify \diffcs $d$ on $M$ such that $d(0)\in M\backslash \{0, 1\}$.
\begin{lemma}
	Let  $d\in \DW_{1}(M)$ such that $d(0)\in M\backslash \{0, 1\}$. Then the following statements hold.
	\begin{enumerate}
		\item $d(x)\in \{d(0), 1\}$	
		for all $x\in M\backslash \{0, d(0), 1\}$.\mlabel{it:5101}
		
		\item  There exists at most
		one element $u$ in $ M\backslash \{0,  1\}$ such that $d(u)=1$.
		\mlabel{it:5102}
		
		\item $d^2(0)\in \{0, d(0)\}$.	\mlabel{it:5103}
		
		\item If $d^2(0)=0$, then $d(1)\not\in \{d(0), 1\}$.	\mlabel{it:5104}
		
		\item If $d(x)=d(0)$ for all $x\in M\backslash \{1\}$, then $d(1)\in \{0, d(0)\}$.	\mlabel{it:5105}
		
		\item  If $|M|\geq 5$, then
		$d(x)=d(0)$ 	for all $x\in M\backslash \{d(0), 1\}$.
		\mlabel{it:5106}
		
		\item  If $|M|\geq 5$, then
		$d(1)\in \{0, d(0)\}$.
		\mlabel{it:5107}
	\end{enumerate}
	\mlabel{lem:510}
\end{lemma}
\begin{proof}
	Let  $d\in \DW_{1}(M)$ such that $d(0)\in M\backslash \{0, 1\}$.	
	
	\mnoindent
	\mref{it:5101}		
	For each $x\in M\backslash \{0, d(0), 1\}$, we have $d(0)\wedge x=0$, and so
	$d(0)=d(0\wedge x)= (d(0)\wedge x)\vee (0\wedge d(x))\vee (d(0)\wedge d(x))=d(0)\wedge d(x)$
	by Eq.\meqref{eq:301}. 
	Thus $d(0)\leq d(x)$, i.e,
	$d(x)=d(0)$ or $d(x)=1$, and hence \mref{it:5101} holds.	
	
	\mnoindent
	\mref{it:5102}
	If there exist $u, v\in M\backslash \{0,  1\}$ such  that $u\neq v$ and $d(u)=d(v)=1$, then $d(0)=d(u\wedge v)= (d(u)\wedge v)\vee (u\wedge d(v))\vee (d(u)\wedge d(v))=1, $
	a contradiction. Thus \mref{it:5102} holds.	
	
	\mnoindent
	\mref{it:5103} If $d(d(0))=v\in M\backslash \{0, d(0), 1\}$, then
	we have  by Eq.\meqref{eq:301} that
	$v=d(d(0))\wedge v\leq d(d(0)\wedge v)=d(0),$
	a contradiction. Thus $d(d(0))\in \{0, d(0), 1\}$.
	
	If $d(d(0))=1$, then we have by Items \mref{it:5101} and \mref{it:5102} that $d(x)=d(0)$ for all  $x\in M\backslash \{0, d(0), 1\}$. Fix $c \in M\backslash \{0, d(0), 1\}$. Then $d(c)=d(0)$, which implies by Eq.\meqref{eq:301} that
	$$d(0)=d(d(0)\wedge c)= (d(d(0))\wedge c)\vee (d(0)\wedge d(c))\vee (d(d(0))\wedge d(c))=c\vee d(c)=1,$$
	a contradiction. Thus $d(d(0))\neq 1$, and hence \mref{it:5103} holds.
	
	\mnoindent
	\mref{it:5104} If $d(d(0))=0$, then we have  by Eq.\meqref{eq:301} that
	$d(1)\wedge d(0)\leq d(1\wedge d(0))=d(d(0))=0,$
	and so $d(1)\wedge d(0)=0$. Thus $d(1)\not\in \{ d(0), 1\}$,  yielding \mref{it:5104}.
	
	\mnoindent
	\mref{it:5105} Assume that  $d(x)=d(0)$ for all $x\in M\backslash \{1\}$. If $d(1)=1$ or $d(1)=v\in M\backslash \{0, d(0), 1\}$.
	Then we have
	by Eq.\meqref{eq:301} that
	$v=d(1)\wedge v\leq d(1\wedge v)=d(v)=d(0),$
	a contradiction. Thus $d(1)\in \{0, d(0)\}$ which implies \mref{it:5105}.
	
	\mnoindent
	\mref{it:5106} Assume that $|M|\geq 5$, and
	there exists $u\in M\backslash \{0, d(0),  1\}$ such  that  $d(u)=1$. Fix $v\in  M\backslash \{0, d(0), u, 1\}$. Then  $u\wedge v=0$, and so 
	$d(0)=d(u\wedge v)= (d(u)\wedge v)\vee (u\wedge d(v))\vee (d(u)\wedge d(v))=v\vee d(v), $
	which implies that $v\leq d(0)$, a contradiction. Thus
	we have by Item \mref{it:5101} that $d(x)=d(0)$ 	for all $x\in M\backslash \{ d(0), 1\}$, that is, \mref{it:5106}  holds.	
	
	\mnoindent
	\mref{it:5107} Assume that $|M|\geq 5$.  If $d(1)=1$ or $d(1)=v\in M\backslash \{0, d(0), 1\}$, then 
	$d(v)=d(0)$ by Item \mref{it:5106}, which implies
	by Eq.\meqref{eq:301} that
	$v=d(1)\wedge v\leq d(1\wedge v)=d(v)=d(0),$
	a contradiction. Thus $d(1)\in \{0, d(0)\}$, giving \mref{it:5107}.	
\end{proof}

\begin{theorem}
	Let $M$ be a quasi-chain with $|M|\geq 5$. Then all \diffc $d$  on $M$ such that $d(0)=a\in M\backslash \{0, 1\}$ are given as follows.
	\begin{enumerate}
		\item $\textbf{C}_{(a)}$, defined in Example \mref{exa:301} \mref{it:3011e};
		
		\item  $(\textbf{C}_{(a)})^{0}$, defined in Eq.\meqref{eq:du};
		
		\item  $\lambda_{a}$, defined in Lemma \mref{lem:511}.
	\end{enumerate}	
	
	\mlabel{pro:503}
\end{theorem}
\begin{proof}
	Assume that $M$ is a quasi-chain with $|M|\geq 5$. 
	 We know  by Example \mref{exa:301} \mref{it:3011e}, Proposition \mref{por:111} and Lemma \mref{lem:511} that $\textbf{C}_{(a)}$, $(\textbf{C}_{(a)})^{0}$ and $\lambda_{a}$ are all in $\DW_{1}(M)$.
	
	Conversely,	let $d\in \DW_{1}(M)$ such that  $d(0)=a\in M\backslash \{0, 1\}$. Then we have by Lemma
	\mref{lem:510} \mref{it:5103}, \mref{it:5106} that
	$d(a)\in \{0, a\}$ and $  d(x)=a$ for all $x\in M\backslash \{a, 1\}$.
	
	If $d(a)=a$, then $d(y)=a$ for all $y\in M\backslash \{ 1\}$,  so $d$  equals 
	$ \textbf{C}_{(a)}$ or $ (\textbf{C}_{(a)})^{0}$ by Lemma
	\mref{lem:510} \mref{it:5105}.
	
	If $d(a)=0$, then $d(1)=0$ by Lemma
	\mref{lem:510} \mref{it:5104} and  \mref{it:5107}. Thus
	$d=\lambda_{a}$.
\end{proof}

\subsubsection{Enumeration of the \diffcs}
\mlabel{sss:enum}

At the end of the paper, we enumerate \diffcs on finite quasi-antichain in Theorem \mref{the:414}, together with the classification for $M_2$. 

First recall that in \cite{GG}, we prove that
\begin{lemma}\cite[Theorem 3.21]{GG}
	Let $M$ be a $n$-element quasi-chain with $n\geq 4$. Then
	$|\DO(M)|=2+\sum_{k=1}^{n-2}(k+1)\tbinom{n-2}{k}$. 	
	\mlabel{lem:512}
\end{lemma}

\begin{theorem}
	\begin{enumerate}
		\item 
Let $M_{2}=\{0, a, b, 1\}$ be the $4$-element quasi-antichain. Then
all $36$ \diffcs on $M_{2}$ are given
as follows.

\begin{enumerate}
	\item $9$ derivations;	
		
	\item $9$ \diffcs on $M_{2}$
	 such that
	$d(0)=a$, which include
	 $\textbf{C}_{(a)} ~($ in Example \mref{exa:301} \mref{it:3011e}$)$,
	 $(\textbf{C}_{(a)})^{0}$,
	 $\lambda_{a}$ (in Lemma \mref{lem:511}),
	 $\psi_{(a)}$ (in Theorem \mref{the:001}),
	 	$(\psi_{(a)})^{u}$ for all $u\in M_{4}\backslash \{1\}$,
	 $\Phi_{a}$  (in Lemma \mref{lem:513}),
	 $(\Phi_{a})^{0}$;	
	\item $9$ \diffcs on $M_{2}$
	such that
	$d(0)=b$ in the same way as $($b$)$;
	\item  $9$ \diffcs on $M_{2}$
	such that
	$d(0)=1$, which include
	$\textbf{C}_{(1)}~ ($ in Example \mref{exa:301} \mref{it:3011e}$)$,
$(\textbf{C}_{(1)})^{u}$ for all $u\in M_{4}\backslash \{1\}$, 		
$ \phi_{1}=\left( \begin{matrix}
0   & a    & b & 1 \\
1   &  b   & 1 &  b
\end{matrix}
\right)$, $ \phi_{2}=\left( \begin{matrix}
0   & a    & b & 1 \\
1   &  1   & a &  a
\end{matrix}
\right)$, $(\phi_{1})^{0}, (\phi_{2})^{0}$ and $ \phi_{3}=\left( \begin{matrix}
0   & a    & b & 1 \\
1   &  b   & a &  0
\end{matrix}
\right)$.
\end{enumerate}
\mlabel{it:51-1}
\item  For $n\geq 5$,  $|\DW_{1}(M_{n-2})|=\dfrac{3n^{3}-22n^{2}+61n-50}{2}+\sum_{k=1}^{n-2}(k+1)\tbinom{n-2}{k}.$
\mlabel{it:51-2}
	\end{enumerate}
\mlabel{the:414}
\end{theorem}

\begin{proof} 
\mnoindent
\mref{it:51-1}	Let $d\in \DW_{1}(M_{2})$. 
If $d(0)=0$, then $d$ is a derivation by Proposition \mref{pro:500}, and so 
$d$ has $9$ choices by Lemma \mref{lem:512}.

If $d(0)=a$, then	since $M_{2}$ is a distributive lattice, we know by Example \mref{exa:301} \mref{it:3011e}, Proposition \mref{por:111}, Lemma \mref{lem:511}, Theorem \mref{the:001} and Lemma \mref{lem:513} that $\textbf{C}_{(a)}$, $\textbf{C}_{(a)}^{0}$,  $\lambda_{a}$,  $\psi_{(a)}$, $(\psi_{(a)})^{u}$, $\Phi_{a}$ and $(\Phi_{a})^{0}$	 are all in $\DW_{1}(M_{2})$.
Now we  claim that 	 $d$ is  of the nine forms. 
In fact, we have 	 $d(b)\in \{a, 1\}$ and
$d(a)\in \{0, a\}$ by Lemma~\mref{lem:510} \mref{it:5101} and \mref{it:5103}. Consider the following cases.

\begin{enumerate}
	\item[$\textcircled{1}$.]  
	$d(a)=d(b)=a$. Then $d(y)=a$ for all $y\in M\backslash \{ 1\}$, and so $d$ is equal to
$ \textbf{C}_{(a)}$ or $ (\textbf{C}_{(a)})^{0}$ by Lemma
\mref{lem:510} \mref{it:5105}.

	\item[$\textcircled{2}$.] 
	$d(a)=a$ and $d(b)=1$.  Since $M_{2}$ is distributive,
$d$
is equal to  $\psi_{(a)}$
or $ (\psi_{(a)})^{u}$ for some $u\in M_{2}\backslash \{1\}$ by
Theorem \mref{the:001} and Proposition \mref{por:111}.

\item[$\textcircled{3}$.]  $d(a)=0$ and $d(b)=1$. Since $d(1)\in \{0, b\}$ by Lemma
	\mref{lem:510} \mref{it:5104}, we get	$d$ is equal to $\Phi_{a}$ or $(\Phi_{a})^{0}$
	by Lemma \mref{lem:513} and Proposition \mref{por:111}.
	
		\item[$\textcircled{4}$.]  $d(a)=0$ and $d(b)=a$.  Since $d(1)\in \{0, b\}$ by Lemma
		\mref{lem:510} \mref{it:5104},
		we get $d$ is equal to $\lambda_{a}$ or
		$\Theta$.
		But $\Theta\not\in \DW_{1}(M_{2})$ by Lemma \mref{lem:513}, hence $d=\lambda_{a}$
		by Lemma \mref{lem:511}.	
\end{enumerate}

If $d(0)=b$, then similarly $d$ has $9$ choices.
	
If $d(0)=1$, then, by Theorem \mref{pro:502}, Proposition \mref{por:111} and Lemmas \mref{lem:501}, \mref{lem:506}, $d$ is of the nine forms:
 $\textbf{C}_{(1)}$,  
 $(\textbf{C}_{(1)})^{u}$ for  $u\in M_{2}\backslash \{1\}$, 
 $ \beta_{b}=\phi_{1}=\left( \begin{matrix}
 0   & a    & b & 1 \\
 1   &  b   & 1 &  b
 \end{matrix}
 \right)$ (as in Lemma \mref{lem:501} \ref{it:5011}), $ \phi_{2}=\left( \begin{matrix}
 0   & a    & b & 1 \\
 1   &  1   & a &  a
 \end{matrix}
 \right)$, $(\phi_{1})^{0}=\theta_{a\to b}$, $ (\phi_{2})^{0}=\theta_{b\to a}$ and $ \alpha^{u \leftrightarrow v}=\phi_{3}=\left( \begin{matrix}
 0   & a    & b & 1 \\
 1   &  b   & a &  0
 \end{matrix}
 \right)$.	

\mnoindent
\mref{it:51-2} Let $n\geq 5$ and  $d\in\DW_{1}(M_{n-2})$. If $d(0)=0$, then $d$ has $2+\sum_{k=1}^{n-2}(k+1)
\tbinom{n-2}{k}$ choices by Proposition \mref{pro:500} and Lemma \mref{lem:512}.	

If $d(0)=d(1)=1$, then $d=\textbf{C}_{(1)}$ by Proposition \mref{pro:001}, and so $d$ has only one choice.
If $d(0)=1$, $ d(1)\in M_{n-2}\backslash \{0, 1\}$, then $d$ has $(n-2)(1+(n-3)^{2}+\tbinom{n-3}{2})$
choices by Theorem \mref{pro:502}.
 If $d(0)=1, d(1)=0$, then $d$ has $1+(n-2)(n-3)+\tbinom{n-2}{2}$
 choices
 by Theorem \mref{pro:502}.
  
 If  $d(0)\in M_{n-2}\backslash \{0, 1\}$, then $d$ has $3(n-2)$
 choices by Theorem \mref{pro:503}.

Thus  when $n\geq 5$, we get
\begin{eqnarray*}
&&|\DW_{1}(M_{n-2})|\\
&=&2+\sum_{k=1}^{n-2}(k+1)\tbinom{n-2}{k}+1+(n-2)\Big(1
 +(n-3)^{2}+\tbinom{n-3}{2}\Big)\\
 &&+1+(n-2)(n-3)+\tbinom{n-2}{2}+3(n-2)\\
&=&\dfrac{1}{2}(3n^{3}-22n^{2}+61n-50)+\sum_{k=1}^{n-2}(k+1)\tbinom{n-2}{k}.  \hspace{5.5cm} \qedhere 
\end{eqnarray*}
\end{proof}

 We finally list $|\DW_{1}(M_{n-2})|$ for $4\leq n \leq 11$ in the following table.
 \begin{center}
 	\begin{tabular}{|c|c|c|c|c|c|c|c|c|}
 		\hline $n$ & 4 &5 & 6 & 7 &8 & 9 & 10&11\\
 		\hline $\dfrac{1}{2}(3n^{3}-22n^{2}+61n-50)$ &  &40& 86 & 164 &283 & 452 & 680&976\\
 			\hline $\sum_{k=1}^{n-2}(k+1)\tbinom{n-2}{k}$ &  &19& 47 & 111 &255 & 575 & 1279&2815\\
 				\hline $|\DW_{1}(M_{n-2})|$ & 36 &59& 133 & 275 &538 & 1027 & 1959&3791\\\hline
 	\end{tabular}
 \end{center}
 This integer sequence $36, 59,133,275,538,1027,1959,3791,\cdots$ is not in the database OEIS~\mcite{OEIS}.

\mnoindent
{\bf Acknowledgments.}
This work is supported by the NSFC Grants (Nos. 11801239, 12171022).
We thank
the anonymous referees for their helpful suggestions.
\delete{
\smallskip 

\noindent
{\bf Conflict of Interest.} The authors have no competing interests to declare that are relevant to the content of this article.
}
\smallskip

\noindent
{\bf Data Availability Statement.} No data are associated with this article. 

%The data used to support the findings of this study are included within the article.

\end{document}